\documentclass{amsart}
\usepackage{latexsym, graphicx, amsmath, amssymb, cite}

\newtheorem{theorem}{Theorem}[section]
\newtheorem{lemma}[theorem]{Lemma}
\newtheorem{proposition}[theorem]{Proposition}
\newtheorem{corollary}[theorem]{Corollary}
\newtheorem{conjecture}[theorem]{Conjecture}

\theoremstyle{definition}
\newtheorem{definition}[theorem]{Definition}

\theoremstyle{remark}
\newtheorem{remark}[theorem]{Remark}

\numberwithin{equation}{section}

\newcommand{\R}{{\mathbb R}}

\newcommand{\Z}{{\mathbb Z}}

\newcommand{\T}{{\mathbb T}}

\newcommand{\supp}{{\operatorname{supp}}}
\newcommand{\Hess}{{\operatorname{Hess}}}

\newcommand{\rank}{{\operatorname{rank}}}
\newcommand{\II}{{\operatorname{I\hspace{-0.05em}I}}}
\newcommand{\lap}[1]{{\sqrt{-\Delta_{#1}}}}
\newcommand{\tg}{{\tilde g}}
\newcommand{\inj}{{\operatorname{inj}}}
\newcommand{\vol}{{\operatorname{vol}}}
\newcommand{\bfk}{{\mathbf k}}

\begin{document}

\title{Period Integrals in Nonpositively Curved Manifolds}


\author{Emmett L. Wyman}
\address{Johns Hopkins University}
\curraddr{}
\email{ewyman3@math.jhu.edu}
\thanks{This article was supported in part by NSF grant DMS-1665373.}


\date{}
\dedicatory{}

\begin{abstract} We provide an improvement of a half power of log to standard bounds on integrals of Laplace eigenfunctions over submanifolds of codimension $2$ or more, where the ambient space is a compact Riemannian manifold with negative sectional curvature. We provide the same improvement for hypersurfaces whose second fundamental form differs sufficiently from that of spheres of infinite radius. This result extends those obtained in the $2$-dimensional setting by ~\cite{CSPer,Gauss,emmett1,emmett2}.
\end{abstract}

\maketitle


\section{Introduction}

\subsection{Background}

Let $(M,g)$ be a compact $n$-dimensional Riemannian manifold without boundary. Let $\Delta_g$ denote the Laplace-Beltrami operator, written in local coordinates as
\[
	\Delta_g = |g|^{-1/2} \sum_{i,j} \partial_i(|g|^{1/2} g^{ij} \partial_j).
\]
Let $e_j$ for $j = 0,1,2,\ldots$ form a Hilbert basis of eigenfunctions of $\Delta_g$ with corresponding eigenvalues $\lambda_j$, i.e.
\[
	-\Delta_g e_j = \lambda_j^2 e_j.
\]
We are interested in the relationship between the geometry of $M$ and asymptotic bounds on the means of eigenfunctions over submanifolds as the eigenvalue tends to infinity.

This class of problems has its roots in the theory of automorphic forms, where bounds on the Fourier coefficients of eigenfunctions along closed geodesics are of interest. Using Kuznecov sum formulae, Good ~\cite{Good} and Hejhal ~\cite{Hej} independently obtained bounds
\[
	\int_\gamma e_j \, ds = O(1)
\]
where $M$ is a compact hyperbolic surface and $\gamma$ a closed geodesic. Later Zelditch ~\cite{ZelK} extended this result to the general Riemannian setting and obtained a Kuznecov sum formula
\begin{equation} \label{zelditch}
	\sum_{\lambda_j \leq \lambda} \left| \int_\Sigma e_j \, d\sigma \right|^2 \sim \lambda^{n-d} + O(\lambda^{n-d-1})
\end{equation}
for $d$-dimensional submanifolds $\Sigma$, where $d\sigma$ is the surface element on $\Sigma$. This provides the general bound,
\begin{equation} \label{general bound}
	\left| \int_\Sigma e_j \, d\sigma \right| = O(\lambda_j^\frac{n-d-1}{2}),
\end{equation}
which is optimal on the sphere for \emph{any}\footnote{Indeed, by \eqref{zelditch} and the fact that the gaps between successive distinct eigenvalues on $S^n$ approach a constant. See \cite{emmett3} for a detailed argument.} submanifold $\Sigma$.

In ~\cite{Rez}, Reznikov extended the bounds in ~\cite{Good,Hej} to geodesic circles and closed horocycles in hyperbolic surfaces of finite area, and put forth a conjecture for optimal bounds.

\begin{conjecture}[~\cite{Rez}] \label{reznikov conjecture} Let $M$ be a compact hyperbolic surface and $\gamma$ a closed geodesic or geodesic circle. Then,
\[
	\left| \int_\gamma e_j \, ds \right| = O(\lambda_j^{-1/2 + \epsilon})
\]
for all $\epsilon > 0$.
\end{conjecture}

\noindent It seems the standard techniques will only yield improvements by a power of $\log \lambda_j$ over the standard bounds. The conjecture, let alone any polynomial improvement over the standard bounds, seems a long way off.

The first improvement on \eqref{general bound} was obtained by Chen and Sogge ~\cite{CSPer}, who used the Gauss-Bonnet theorem to show
\[
	\left| \int_\gamma e_j \, ds \right| = o(1)
\]
where $M$ is a compact surface with negative sectional curvature and $\gamma$ is a geodesic. This result was later improved by Sogge, Xi, and Zhang ~\cite{Gauss} by providing an explicit decay of $O(1/\sqrt{\log \lambda_j})$ under some weaker sectional curvature hypotheses. The author ~\cite{emmett1} extended ~\cite{CSPer} and later ~\cite{Gauss} from geodesics to a wide class of curves satisfying some curvature conditions, albeit without the weakened sectional curvature hypotheses. The result is summarized below.

\begin{theorem}[\cite{emmett2}] \label{wyman theorem} Let $M$ be a compact Riemannian surface without boundary with nonpositive sectional curvature. For each $p \in M$ and $v \in S_pM$, let $\bfk(v)$ denote the limit of the curvature of the circular arc at $p$ with center taken out to infinity along the geodesic ray in direction $v$. Then, if $\gamma$ is a closed curve in $M$ such that
\[
	\kappa_\gamma \neq \bfk(v) \qquad \text{ for all normal vectors $v$ to $\gamma$,}
\]
then
\[
	\int_\gamma e_j \, ds = O(1/\sqrt{\log \lambda_j}),
\]
where here $\kappa_\gamma$ denotes the geodesic curvature of $\gamma$.
\end{theorem}

\noindent For a more detailed definition of the limiting curvature $\bfk$, see ~\cite{emmett1} and ~\cite{emmett2}.

There have been a number of recent improvements on the general bounds assuming some dynamical properties of the geodesic flow. Canzani, Galkowski, and Toth ~\cite{toth} provided a little-$o$ improvement on bounds on averages of Cauchy data over hypersurfaces of eigenfunctions belonging to a sequence with defect measure. The author ~\cite{emmett3} provided a little-$o$ improvement on \eqref{general bound} provided the set of looping directions which depart from and arrive at $\Sigma$ conormally has measure zero. Using quantum defect measures, Canzani and Galkowski ~\cite{CanGal} recently obtained the little-$o$ improvement for a vast range of situations containing results in ~\cite{toth,emmett3}.

\subsection{Statement of Results}

In this article, we provide a generalization of ~\cite{emmett2} to nonpositively curved manifolds of arbitrary dimension. Our first result provides an improvement of $1/\sqrt{\log \lambda_j}$ to \eqref{general bound} if $M$ has negative sectional curvature and $\Sigma$ has codimension at least $2$.

\begin{theorem}\label{main 1} Let $(M,g)$ be a compact, $n$-dimensional Riemannian manifold, without boundary, with negative sectional curvature. Let $\Sigma$ be a closed $d$-dimensional submanifold with $d \leq n-2$. Then,
\[
	\left| \int_\Sigma e_j \, d\sigma \right| = O(\lambda_j^\frac{n-d-1}{2}/\sqrt{\log \lambda_j}),
\]
where $d\sigma$ denotes the induced measure on $\Sigma$.
\end{theorem}

Our second result treats the codimension $1$ case and requires a generalization of the limiting curvature $\bfk$ from the two-dimensional case.

\begin{definition}\label{limiting curvature}
Fix $p \in M$ and a unit vector $v \in S_pM$ and consider a unit-speed geodesic 
\begin{align*}
		\gamma : \R &\to M \\
			r &\mapsto \gamma(r)
\end{align*}
with $\gamma'(0) = v$. We denote by $\II_{H(v)}$ the quadratic form on $N_p\gamma$ defined as follows. Fix $X \in N_p\gamma$ and let $J_X$ denote the unique Jacobi field for which $J_X(0) = X$ and
\[
	|J_X(r)| = O(1) \qquad \text{ for } r \geq 0.
\]
Then, set
\[
	\II_{H(v)}(X,X) = \left\langle - \frac{D}{dr} J_X(0) , X \right\rangle v.
\]
$\II_{H(v)}$ extends to a vector-valued symmetric bilinear form on $N_p\gamma$ by the usual trick,
\[
	\II_{H(v)}(X,Y) = \frac{1}{2} (\II_{H(v)}(X+Y,X+Y) - \II_{H(v)}(X,X) - \II_{H(v)}(Y,Y)).
\]
\end{definition}

We prove that $\II_{H(v)}$ is well-defined and continuous in $v$, and that $\langle \II_{H(v)}, v\rangle$ is positive semidefinite in Proposition \ref{limiting curvature prop}. The geometric meaning of $\II_{H(v)}$ is clearer in the universal cover. By the Cartan-Hadamard theorem, we identify the universal cover $\tilde M$ with $\R^n$ and consider a lift $\tilde p \in \tilde M$ of $p$ and a lift $\tilde v \in S_{\tilde p} \tilde M$ of $v$. Let $H(\tilde v)$ denote the hyperpersurface obtained as a limit of the spheres intersecting $\tilde p$ with centers taken out along $\gamma(r)$ as $r \to -\infty$. Then, $\II_{H(v)}$ coincides with the second fundamental form of $H(\tilde v)$. Note that $H(\tilde v)$ are exactly the \emph{horospheres} in the hyperbolic setting, and we will use the same name even if $M$ has nonconstant curvature.

Our second main theorem, which pertains to period integrals over hypersurfaces, requires hypotheses on the quadratic forms $\langle \II_\Sigma - \II_{H(v)}, v \rangle$ on $T\Sigma$ for each unit normal vector $v$.

\begin{theorem} \label{main 2} Let $(M,g)$ be as in Theorem \ref{main 1} except allow $M$ to have nonpositive sectional curvature, and let $\Sigma$ be a closed hypersurface. If
\begin{equation} \label{main 2 hypotheses}
	\rank(\langle \II_\Sigma - \II_{H(v)}, v \rangle ) + \rank(\langle \II_\Sigma - \II_{H(-v)}, -v \rangle ) \geq n \qquad \text{ for each } v \in SN\Sigma,
\end{equation}
then
\begin{equation} \label{main 2 bound}
	\int_\Sigma e_j \, d\sigma = O(1/\sqrt{\log \lambda_j}).
\end{equation}
\end{theorem}

\begin{remark}\label{generalizations} The results of Theorems \ref{main 1} and \ref{main 2}, and of Corollary \ref{criteria} to follow, still hold if we replace the eigenfunctions by quasimodes $\Psi_\lambda$ with $\|\Psi_\lambda\|_{L^2} \leq 1$ and with spectral support on the frequency band $[\lambda, \lambda + \frac{1}{\log \lambda}]$. The submanifold $\Sigma$ need not be closed, either, provided the surface element $d\sigma$ is multiplied by some smooth, compactly supported cutoff. This will be made apparent in the next section.
\end{remark}

The arguments in Section \ref{GEOMETRY} allow us to pick out some criteria for hypersurfaces which satisfy the hypotheses of Theorem \ref{main 2}. As a consequence, we have the following corollaries. (See Proposition \ref{curvature properties} and Remark \ref{curvature remark} for details.)

\begin{corollary} \label{criteria}
	Let $M$ and $\Sigma$ be as in Theorem \ref{main 2}. Then $\Sigma$ satisfies \eqref{main 2 hypotheses} and hence \eqref{main 2 bound} if any of the following hold.
	\begin{enumerate}
		\item At each point in $\Sigma$, at least $n/2$ of the principal curvatures of $\Sigma$ lie outside the interval $[a,b]$, where $0 \leq a \leq b$ are constants such that
		\[
			0 \geq -a^2 \geq K \geq -b^2
		\]
		on $M$, where $K$ is the sectional curvature of $M$.
		\item $\Sigma$ is a geodesic sphere in $M$.
		\item $M$ has strictly negative curvature and $\Sigma$ is a totally geodesic hypersurface.
	\end{enumerate}
\end{corollary}

Note Theorem \ref{main 2} not only generalizes Theorem \ref{wyman theorem} to hypersurfaces of arbitrary manifolds, but it is stronger even in the two-dimensional case. In Theorem \ref{main 2}, the curvature of $\gamma$ is \emph{signed}, and in Theorem \ref{wyman theorem} it is not. This lets us apply Theorem \ref{main 2} to all spheres, not just those of some bounded radius (for reference, see ~\cite[Corollary 1.6]{emmett2}).

Section 2 is dedicated to reducing Theorems \ref{main 1} and \ref{main 2} to bounds on a kernel involving the half wave operator. Following this, we lift our computation to the universal cover $(\tilde M,\tg)$ of $M$, which we identify with $\R^n$ by the Cartan-Hadamard theorem as in ~\cite{Berard,CSPer,Gauss,emmett2}. We then rephrase the kernel as a sum of kernels over the group of deck transformations $\Gamma$ associated with the covering map. Section 3 is dedicated to writing these summands as oscillatory integrals, roughly 
\begin{equation}\label{intro osc int}
	\sum_{\alpha \in \Gamma} \int_\Sigma \int_\Sigma a_\alpha(x,y) e^{i\lambda \phi_\alpha(x,y)} \, d\sigma(x) \, d\sigma(y)
\end{equation}
with phase function
\[
	\phi_\alpha(x,y) = d_\tg(\alpha \tilde x, \tilde y)
\]
where $\tilde x$ and $\tilde y$ are respective lifts of $x$ and $y$ to $\tilde M$, and where $d_\tg$ denotes the distance function on the universal cover $\tilde M$. Section 4 is dedicated to computing and bounding derivatives of the phase function so that we can use the method of stationary phase in Section 5. Theorems \ref{main 1} and \ref{main 2} follow if we can show each of the non-identity terms of \eqref{intro osc int} is $O(\lambda^{-d/2})$ and $O(\lambda^{-n/2})$, respectively, where the constants implied by the big-$O$ notation are sufficiently uniform.

\subsection{Examples and Limitations of Theorems \ref{main 1} and \ref{main 2}} There are two examples of manifolds which help to illustrate Theorems \ref{main 1} and \ref{main 2}: the flat torus and a compact hyperbolic manifold. These two examples help to motivate the statements of Theorems \ref{main 1} and \ref{main 2}. At the same time, these specific examples show the deficiencies of our main results and suggest that a more complete picture must be provided with other methods.\\

\noindent \textbf{The Torus.} Suppose $M$ is the flat torus $\T^n = \R^n/2\pi \Z^n$. According to Theorem \ref{main 2}, since $\T^n$ is flat, we should obtain a decay of $O(1/\sqrt{\log \lambda_j})$ on integrals of eigenfunctions $e_j$ of the torus over hypersurfaces $\Sigma$, provided $\Sigma$ has at least $n/2$ nonzero principal curvatures at each point. In ~\cite{HGtorus}, Hezari and Riviere present the following result on the torus, which has both stronger hypotheses and (much) stronger bounds than those in Theorem \ref{main 2}.

\begin{theorem}[\cite{HGtorus}]
	Let $\Sigma$ be a smooth, compact, embedded, oriented hypersurface of $\T^n$ without boundary with surface measure $\sigma$, and suppose all principal curvatures of $\Sigma$ are nonzero at each point on $\Sigma$. Then,
	\[
		\int_\Sigma e_\lambda \, d\sigma = O(\lambda^{-1/2 + \epsilon})
	\]
	where $\epsilon$ is any positive constant, but which is allowed to vanish when $n \geq 5$.
\end{theorem}

\noindent Hezari and Riviere explain that the problem comes down to counting lattice points on spheres, for which there are bounds
\begin{equation}\label{lattice point bound}
	\#\{ m \in \Z^n : |m| = \lambda\} = O(\lambda^{n-2 + \epsilon}) \qquad \text{ for all } \epsilon > 0
\end{equation}
but where $\epsilon$ is allowed to vanish when $n \geq 5$. At the same time if $\Sigma$ is a rational hyperplane in $\T^n$, one may pick out a sequence of exponentials with eigenvalues tending to infinity whose restrictions to $\Sigma$ are a constant. In this sense, some nonvanishing curvature conditions on $\Sigma$ are necessary to obtain decay.

Now consider the situation where $d \leq n-2$. Theorem \ref{main 1} requires that the sectional curvature of $M$ be strictly negative. However, it is reasonable to ask if a similar result applies in the flat setting. Consider the specific case where $\Sigma = \T^d \times \{0\}^{n-d}$ for $d \leq n-2$. By writing $e_\lambda$ as a linear combination of exponentials and using Cauchy-Schwarz, we obtain
\[
	\left| \int_\Sigma e_\lambda \, d\sigma \right| \leq \left( \sum_{|m| = \lambda} |\hat \sigma(m)|^2 \right)^{1/2}
\]
where
\[
	\hat \sigma(m) = \int_{\T^d} e^{-i\langle x, m' \rangle} \, dx
	= \begin{cases}
		(2\pi)^d & m' = 0 \\
		0 & m' \neq 0.
	\end{cases}
\]
where $m' = (m_1,\ldots, m_d)$ are the first $d$ coordinates of $m$. This and \eqref{lattice point bound} yields
\[
	\left| \int_\Sigma e_\lambda \, d\sigma \right| = O(\lambda^\frac{d-2+\epsilon}{2}) \qquad \text{ for all } \epsilon > 0,
\]
which is better than the bound in Theorem \ref{main 1}. Though this computation applies only to a few specific submanifolds, it suggests that Theorem \ref{main 1} may apply to $M$ with merely nonpositive sectional curvature. However, this result is inaccessible with the methods used to prove of our main results. Indeed, we require negative curvature to obtain a uniform constant and finish the proof of Theorem \ref{main 1} in Chapter 5.
\\

\noindent \textbf{A Compact Hyperbolic Manifold.} Suppose $M$ is a compact hyperbolic manifold, i.e. the sectional curvature is identically $-1$. By the corollary, Theorem \ref{main 2} requires that at least $n/2$ of the principal curvatures of $\Sigma$ not be $1$. We ask whether we require the full conditions on $\Sigma$ to obtain the improved bound \eqref{main 2 bound}. Consider the extreme example where $\Sigma$ is precisely a horosphere\footnote{There are no closed horospheres in a compact hyperbolic manifold. However by Remark \ref{generalizations}, it suffices to consider a small, embedded piece of a horosphere with surface measure $\sigma$ which as been multiplied by some smooth, compactly supported bump function.} in $M$. Is the standard bound
\[
	\int_\Sigma e_\lambda \, d\sigma = O(1)
\]
sharp like it is for rational hyperplanes in the torus? A recent result by Canzani, Galkowski, and Toth ~\cite{toth} shows that if $e_\lambda$ is a quantum ergodic sequence, its average over \emph{any} hypersurface will be $o(1)$. If we assume the quantum unique ergodicity conjecture, the standard $O(1)$ bound is never sharp regardless of the conditions on the curvature of $\Sigma$. Whether or not we obtain an explicit decay of $O(1/\sqrt{\log \lambda})$ for period integrals over horospheres is an open question.\\

\noindent \textbf{Acknowledgements.} The author would like to thank his advisor, Chris Sogge, for his ongoing support.


\section{A Standard Reduction and Lift to the Universal Cover}

The following reduction is part of the standard strategy for many problems dealing with the asymptotic distributions of eigenfunctions on manifolds (e.g. ~\cite{Berard,ZelK,SZDuke,STZ,CSPer,Gauss} and many others). We follow the example of ~\cite{SZDuke,STZ} and use pseudodifferential operators to microlocalize to cones in $T^*M$ with small support. Afterwards we perform a lift to the universal cover as in ~\cite{Berard,CSPer,Gauss}.

For both the situations in Theorems \ref{main 1} and \ref{main 2}, we will show
\begin{equation}\label{reduction 1}
	\sum_{\lambda_j \in [\lambda, \lambda + T^{-1}]} \left| \int_\Sigma e_j \, d\sigma \right|^2 \lesssim T^{-1} \lambda^{n-d-1} + e^{CT} \lambda^\delta
\end{equation}
where $\delta$ is some exponent less than $n-d-1$ and where we set
\begin{equation} \label{T = c log lambda}
	T = c \log \lambda
\end{equation}
for some sufficiently small $c$.

Now we introduce Fermi-type coordinates about $\Sigma$. Parametrize a small neighborhood in $\Sigma$ with geodesic normal coordinates $x' = (x_1,\ldots, x_d) \in \R^d$. Then take a smooth, orthonormal frame $v_{d+1},\ldots, v_n$ of the normal bundle of $\Sigma$. Writing $x = (x',x^\perp) \in \R^n$ where $x^\perp = (x_{d+1}, \ldots, x_n)$ are the remaining $n-d$ coordinates, the coordinate map
\begin{equation}\label{fermi coordinates}
	(x',x^\perp) \mapsto \exp(x_{d+1}v_{d+1}(x') + \cdots + x_n v_n(x'))
\end{equation}
parametrizes a small neighborhood in $M$ containing a piece of $\Sigma$. By construction,
\begin{equation}\label{fermi local metric}
	g(x',0) = \begin{bmatrix}
		g_\Sigma(x') & 0 \\
		0 & I_{n-d}
	\end{bmatrix}
\end{equation}
where $g_\Sigma$ is the intrinsic metric on $\Sigma$ and $I_{n-d}$ is the $(n-d) \times (n-d)$ identity matrix. We also note for future use that
\[
	g_\Sigma(x') = I_d + O(|x'|^2)
\]
and that the Christoffel symbols associated with the Levi-Civita connection intrinsic to $\Sigma$ vanish at $x' = 0$ ~\cite{doCarmo}. In particular, we can take the Christoffel symbols to be as small as desired by shrinking the neighborhood parametrized by our coordinates.

Take a finite open cover of $\Sigma$ in $M$ of such coordinate charts and with it a subordinate partition of unity
\[
	\sum_i b_i \equiv 1
\]
on $\Sigma$. By the Cauchy-Schwarz inequality,
\[
	\left| \int_\Sigma e_j \, d\sigma \right|^2 = \left| \sum_i \int_\Sigma b_i e_j \, d\sigma \right|^2 \leq C_\Sigma \sum_i \left| \int_\Sigma b_i e_j \, d\sigma \right|^2,
\]
and so \eqref{reduction 1} follows if we can show
\begin{equation}\label{reduction 2}
	\sum_{\lambda_j \in [\lambda, \lambda + T^{-1}]} \left| \int_\Sigma b e_j \, d\sigma \right|^2 \lesssim T^{-1} \lambda^{n-d-1} + e^{CT} \lambda^\delta
\end{equation}
where $b$ is a smooth function on $\Sigma$ with controllably small support, and the constants in the bounds are allowed to depend on $b$. We will take this a step further and microlocalize to small cones in $T^*M$. Take a partition of unity
\[
	\sum_i a_i \equiv 1
\]
of the sphere $S^{n-1} \subset \R^n$, and take smooth bump functions $\beta_0$ and $\beta_1$ both supported on a small interval in $\R$ and for which $\beta_0 \equiv 1$ near $0$ and $\beta_1 \equiv 1$ near $1$. For each $i$, we define operators\footnote{The purpose of the operator $B_\lambda$ is to filter out geodesics which depart $y$ and arrive at $x$ in sufficiently differing directions, as Lemma \ref{kernel alpha lemma} will show in the next section. This strategy was used before by Sogge, Toth, and Zelditch ~\cite{STZ} who obtained improved sup-norm estimates for eigenfunctions on manifolds provided that, at each point, the set of recurrent directions of geodesics has measure zero.}
\begin{equation}\label{B op}
	B_{i,\lambda} f(x) = \frac{1}{(2\pi)^n} \int_{\R^n} \int_{\R^n} e^{i\langle x - y, \xi \rangle} B_{i,\lambda}(x,y,\xi) f(y) \, dy \, d\xi
\end{equation}
with symbol
\[
	B_{i,\lambda}(x,y,\xi) = \beta_0(|x - y|) \beta_0(|x^\perp|) b(x') \beta_1(|\xi|/\lambda) a_i(\xi/|\xi|),
\]
and similarly
\[
	R_{\lambda} f(x) = \frac{1}{(2\pi)^n} \int_{\R^n} \int_{\R^d} e^{i\langle x - y, \xi \rangle} R(\lambda; x,y,\xi) f(y) \, dy \, d\xi
\]
with symbol
\[
	R_\lambda(x,y,\xi) = \beta_0(|x - y|) \beta_0(|x^\perp|) b(x') (1 - \beta_1(|\xi|/\lambda)).
\]
Note
\[
	\int_\Sigma b e_j \, d\sigma = \sum_i \int_\Sigma B_{i,\lambda} e_j \, d\sigma + \int_\Sigma R_\lambda e_j \, d\sigma.
\]
By the same Cauchy-Schwarz argument as before, \eqref{reduction 2} follows provided we can show
\begin{equation} \label{reduction 3}
	\sum_{\lambda_j \in [\lambda,\lambda+T^{-1}]} \left| \int_\Sigma B_\lambda e_j \, d\sigma \right|^2 \lesssim T^{-1} \lambda^{n-d-1} + e^{CT} \lambda^\delta
\end{equation}
where $B_\lambda$ is defined as in \eqref{B op} with symbol
\begin{equation} \label{B symbol}
	B_\lambda(x,y,\xi) = \beta_0(|x - y|) \beta_0(|x^\perp|) b(x') \beta_1(|\xi|/\lambda) a(\xi/|\xi|)
\end{equation}
where $\beta_0$, $\beta_1$, $a$, and of course $b$ all have adjustably small support, and we can show
\begin{equation} \label{reduction remainder}
	\sum_{\lambda_j \in [\lambda,\lambda+T^{-1}]} \left| \int_\Sigma R_\lambda e_j \, d\sigma \right|^2 = O(\lambda^{-\infty}) \qquad \text{ uniformly for } T \geq 1
\end{equation}
where $R_\lambda$ is as above. The latter bound follows from Cauchy-Schwarz inequality applied to the integral and the following proposition whose proof we defer until the end of the section.

\begin{proposition} \label{remainder prop} Let $R_\lambda$ be as above. Then
\[
	\sup_{x \in \Sigma} \sum_{\lambda_j \in [\lambda, \lambda+1]} |R_\lambda e_j(x)|^2 = O(\lambda^{-\infty}).
\]
\end{proposition}

We will also use the following generalization of the bound \eqref{general bound} to help us contend with \eqref{reduction 3}, whose proof we also defer until the end of the section.

\begin{proposition} \label{intermediate prop}
	Let $B_\lambda$ be as above. Then,
	\[
		\sum_{\lambda_j \in [\lambda, \lambda+1]} \left| \int_\Sigma B_\lambda e_j \, d\sigma \right|^2 = O(\lambda^{n-d-1}).
	\]
\end{proposition}

Let $\chi$ be some nonnegative Schwartz-class function with $\chi(0) = 1$ and $\supp \hat \chi \subset [-1,1]$. Since we can fit some rectangle under the graph of $\chi$, we have \eqref{reduction 3} provided
\begin{equation}\label{reduction 3'}
	\sum_j \chi(T(\lambda_j - \lambda)) \left| \int_\Sigma B_\lambda e_j \, d\sigma \right|^2 \lesssim T^{-1} \lambda^{n-d-1} + e^{CT} \lambda^\delta.
\end{equation}
To access \eqref{reduction 3'}, we will make use of the spectrally-defined half-wave operator,
\[
	e^{it\lap g} = \sum_j e^{it\lambda_j} E_j
\]
where $E_j$ is the orthogonal projection operator onto the $e_j$-th eigenspace. The half-wave operator has kernel
\[
	e^{it\lap g}(x,y) = \sum_j e^{it\lambda_j} e_j(x) \overline{e_j(y)}
\]
and so the kernel of the composition $B_\lambda e^{it\lap g} B_\lambda^*$ is
\[
	B_\lambda e^{it\lap g} B_\lambda^*(x,y) = \sum_j e^{it\lambda_j} B_\lambda e_j(x) \overline{B_\lambda e_j(y)},
\]
where here $B_\lambda^*$ denotes the adjoint of $B_\lambda$.
We use the Fourier inversion formula and the expression above to write the left hand side of \eqref{reduction 3'} as
\begin{align}
	\nonumber \frac{1}{2\pi T} \sum_j &\int_{-\infty}^\infty \hat \chi(t/T) e^{-it\lambda} e^{it\lambda_j} \int_\Sigma \int_\Sigma B_\lambda e_j(x) \overline{B_\lambda e_j(y)} \, d\sigma(x) \, d\sigma(y) \, dt \\
	\label{reduction 3.5} &=\frac{1}{2\pi T} \int_\Sigma \int_\Sigma \int_{-\infty}^\infty \hat \chi(t/T) e^{-it\lambda} B_\lambda e^{it\lap \tg} B_\lambda^*(x,y) \, dt \, d\sigma(x) \, d\sigma(y).
\end{align}

[LEFT OFF HERE]

Let $\beta$ be a smooth bump function on $\R$ such that $\beta(t) = 1$ for $|t| \leq 2$ and $\beta(t) = 0$ for $|t| \geq 3$. At this point we introduce a constant $R$ to be determined later, independent of $T$ and $\lambda$, and dependent only on the geometry of $M$ and $\Sigma$. We cut the integral \eqref{reduction 3.5} into $\beta(t/R)$ and $1 - \beta(t/R)$ parts and obtain
\begin{align}
	\label{reduction beta cut} &= \frac{1}{2\pi T} \int_\Sigma \int_\Sigma \int_{-\infty}^\infty \beta(t/R) \hat \chi(t/T) e^{-it\lambda} B_\lambda e^{it\lap \tg} B_\lambda^*(x,y) \, dt \, d\sigma(x) \, d\sigma(y) \\
	\nonumber &\qquad + \frac{1}{2\pi T} \int_\Sigma \int_\Sigma \int_{-\infty}^\infty (1 - \beta(t/R)) \hat \chi(t/T) e^{-it\lambda} B_\lambda e^{it\lap \tg} B_\lambda^*(x,y) \, dt \, d\sigma(x) \, d\sigma(y).
\end{align}
We let $X_T$ denote the function with Fourier transform $\hat X_T(t) = \beta(t/R) \chi(t/T)$. By reversing our argument, we write the first term in \eqref{reduction beta cut} as
\[
	\frac{1}{T} \sum_j X_T(\lambda_j - \lambda) \left| \int_\Sigma B_\lambda e_j \, d\sigma \right|^2,
\]
which is bounded by a constant multiple $T^{-1}\lambda^{n-d-1}$ by Proposition \ref{intermediate prop} and the fact that
\[
	|X_T(\lambda_j - \lambda)| \leq C_N (1 + |\lambda_j - \lambda|)^{-N} \qquad \text{ for } N = 0,1,2,\ldots
\]
for constants $C_N$ uniform for $T \geq 1$.
Hence, we are done if we can show that
\begin{align}
	\nonumber \left| T^{-1} \int_\Sigma \int_\Sigma \int_{-\infty}^\infty (1 - \beta(t/R)) \hat \chi(t/T) e^{-it\lambda} B_\lambda e^{it\lap \tg} B_\lambda^*(x,y) \, dt \, d\sigma(x) \, d\sigma(y) \right|& \\
	\label{reduction global bound} &\hspace{-4em} \lesssim e^{CT} \lambda^\delta.
\end{align}

As in ~\cite{Berard,CSPer,Gauss}, we will want to replace the half wave operator of \eqref{reduction global bound} with the cosine operator so that we have H\"ugen's principle at our disposal when we lift to the universal cover. By Euler's formula,
\[
	e^{it\lap g} = 2\cos(t\lap g) - e^{-it\lap g},
\]
hence we write what is inside the absolute values in \eqref{reduction global bound} as
\begin{align*}
	\frac{2}{T} & \int_\Sigma \int_\Sigma \int_{-\infty}^\infty (1 - \beta(t/R)) \hat \chi(t/T) e^{-it\lambda} B_\lambda \cos(t\lap g) B_\lambda^*(x,y) \, dt \, d\sigma(x) \, d\sigma(y) \\
	&+ \frac{1}{T} \int_\Sigma \int_\Sigma \int_{-\infty}^\infty (1 - \beta(t/R)) \hat \chi(t/T) e^{-it\lambda} B_\lambda e^{-it\lap g} B_\lambda^*(x,y) \, dt \, d\sigma(x) \, d\sigma(y).
\end{align*}
Setting $\hat X_T(t) = \beta(t/R)\hat \chi(t/T)$ as before and reversing our reduction, the latter term is a constant multiple of
\[
	\sum_j \left(\chi(-T(\lambda_j + \lambda)) - \frac{1}{T} X_T(-(\lambda_j + \lambda)) \right) \left| \int_\Sigma B_\lambda e_j \, d\sigma \right|^2
\]
which vanishes rapidly in $\lambda$ for $T \geq 1$ by Proposition \ref{intermediate prop} and
\[
	|X_T(\tau)| \leq C_N (1 + |\tau|)^{-N} \qquad \text{ for } T \geq 1, \ N = 1,2,\ldots, 
\]
Hence, it suffices to show
\begin{align} \label{reduction 4}
	\nonumber \left| \int_\Sigma \int_\Sigma \int_{-\infty}^\infty (1 - \beta(t/R)) \hat \chi(t/T) e^{-it\lambda} B_\lambda \cos(t\lap g) B_\lambda^*(x,y) \, dt \, d\sigma(x) \, d\sigma(y) \right|&\\
	&\hspace{-4em} \lesssim e^{CT} \lambda^\delta.
\end{align}

We are ready to perform our lift. By the Cartan-Hadamard theorem, we identify the universal cover $\tilde M$ of $M$ with $\R^n$ equipped with the pullback $\tg$ of the metric $g$ through the covering map. Let $\Gamma$ denote the group of deck transformations associated with the covering map and let
\[
	D = \left\{ \tilde x \in \tilde M : d_\tg(\tilde x,0) = \inf_{\alpha \in \Gamma} d_\tg(\alpha \tilde x, 0) \right\}
\]
denote a Dirichlet domain in $\tilde M$ with $0$ chosen to be a lift of a point on $\Sigma$ in the support of $B_\lambda$. Let $\tilde f$ be a smooth, compactly supported function on $\tilde M$ and set
\[
	f(x) = \sum_{\alpha \in \Gamma} \tilde f(\alpha \tilde x)
\]
where $\tilde x$ is any lift of $x$ to $\tilde M$. Since the covering map is a local isometry,
\[
	u(t,x) = \sum_{\alpha \in \Gamma} \cos(t \lap \tg)\tilde f(\alpha \tilde x)
\]
solves the wave equation $(\partial_t^2 - \Delta_g)u = 0$ with initial data $u(0,x) = f(x)$ and $\partial_t u(0,x) = 0$, hence 
\[
	u(t,x) = \cos(t\lap g)f(x).
\]
We conclude
\begin{equation} \label{cosine lift}
	\cos(t\lap g) = \sum_{\alpha \in \Gamma} \alpha^* \cos(t \lap \tg)
\end{equation}
where $\alpha^*$ is the pullback operator through $\alpha$, e.g. $\alpha^* \tilde f(\tilde x) = \tilde f(\alpha \tilde x)$. Hence we will have \eqref{reduction 4} provided
\begin{equation} \label{reduction final}
	\sum_{\alpha \in \Gamma} \left| \int_{\Sigma} \int_{\Sigma} K_\alpha(T,\lambda; x, y) \, d\sigma(x) \, d\sigma(y) \right| \lesssim e^{CT} \lambda^\delta
\end{equation}
where
\begin{equation} \label{conjugated kernel def}
	K_{\alpha}(T,\lambda; x, y) = \int_{-\infty}^\infty (1 - \beta(t/R)) \hat \chi(t/T) e^{-it\lambda} \tilde B_\lambda \alpha^* \cos(t\lap \tg) \tilde B_\lambda^*(\tilde x, \tilde y) \, dt,
\end{equation}
where $\tilde B_\lambda$ is the operator on $\tilde M$ associated with the symbol
\begin{equation} \label{tilde B symbol}
	\tilde B_\lambda(\tilde x, \tilde y, \xi) = \begin{cases}
		B_\lambda(x,y,\xi) & \text{ if } \tilde x, \tilde y \in D, \\
		0 & \text{ otherwise},
	\end{cases}
\end{equation}
and where $\tilde x$ and $\tilde y$ are the respective lifts of $x$ and $y$ to the Dirichlet domain $D$ in the universal cover. We note now for future reference that, by H\"uygen's principle, $K_\alpha(T,\lambda;x,y)$ is supported on $d_\tg(\tilde x, \tilde y) \leq T+1$, after perhaps shrinking the $\tilde x$-support of the symbol $\tilde B_\lambda$. Hence, all except for a finite number of terms in the sum in \eqref{reduction final} is zero. In fact, by volume comparison
\begin{equation}\label{finite sum}
	\#\{\alpha \in \Gamma : \supp K_\alpha(T,\lambda; \ \cdot \ , \ \cdot \ ) \text{ is nonempty} \} = O(e^{CT}).
\end{equation}

This concludes our reduction, but we still need to prove Propositions \ref{remainder prop} and \ref{intermediate prop}. The proof of Proposition \ref{intermediate prop} is very standard but a bit involved, requiring a parametrix of the half wave operator and two consecutive applications of stationary phase. We refer the reader to ~\cite{SZDuke,STZ,Hang,emmett3} for similar arguments.

\begin{proof}[Proof of Proposition \ref{intermediate prop}.] Let $\chi$ be as before, i.e. a nonnegative Schwartz-class function on $\R$ with $\chi(0) = 1$, but now with $\hat \chi$ having adjustably small support. It suffices to show
\begin{equation}\label{intermediate prop 1}
	\sum_j \chi(\lambda_j - \lambda) \left| \int_\Sigma B_\lambda e_j \, d\sigma \right|^2 \lesssim \lambda^{n-d-1}.
\end{equation}
Following the steps in the reduction above, we write \eqref{intermediate prop 1} as
\[
	\left| \int_\Sigma \int_\Sigma \int_{-\infty}^\infty \hat \chi(t) e^{-it\lambda} B_\lambda e^{it\lap g} B_\lambda^*(x,y) \, dt \, d\sigma(x) \, d\sigma(y) \right| \lesssim \lambda^{n-d-1}.
\]
By using H\"ormander's parametrix ~\cite[Chapter 4]{SFIO} or by using the Hadamard parametrix and the arguments in section 5.2.2 of ~\cite{Hang}, we write
\begin{equation} \label{hormander parametrix}
	e^{it\lap g}(x,y) = \int_{\R^n} e^{i(\varphi(x,y,\xi) + tp(y,\xi))} q(t,x,y,\xi) \, d\xi
\end{equation}
modulo a smooth kernel where $q$ is a zero-order symbol in $\xi$ satisfying
\[
	|\partial_\xi^\alpha \partial_{t,x,y}^\beta q(t,x,y,\xi)| \leq C_{\alpha,\beta}(1 + |\xi|)^{-|\alpha|}
\]
for multiindices $\alpha$ and $\beta$, and where since the support of $\hat \chi$ is small, $\hat \chi(t)q(t,x,y,\xi)$ is supported where $d_\tg(x,y)$ is near $0$. After perhaps further restricting the support of $\hat \chi$, the phase function $\varphi$ is defined on the support of $\hat \chi q$, is smooth and homogeneous of degree $1$ in $\xi$, and satisfies
\[
	\varphi(x,y,\xi) = \langle x - y, \xi \rangle + O(|x - y|^2 |\xi|)
\]
where here $x$ and $y$ are written in Fermi coordinates \eqref{fermi coordinates}. Finally,
\[
	p(y,\xi) = \sqrt{\sum_{i,j} g^{ij}(y) \xi_i \xi_j}
\]
is the principal symbol associated with the half-Laplacian $\lap g$.
For $x$ and $y$ in Fermi coordinates,
\begin{align*}
	&\int_{-\infty}^\infty \hat \chi(t) e^{-it\lambda} e^{it\lap g}(x,y) \, dt \\
		&= \int_{\R^n} \int_{-\infty}^\infty \hat \chi(t) q(t,x,y,\xi) e^{i(\varphi(x,y,\xi) + t(p(y,\xi) -\lambda))} \, dt \, d\xi\\
		&= \lambda^n \int_{\R^n} \int_{-\infty}^\infty \hat \chi(t) q(t,x,y,\lambda \xi) e^{i\lambda (\varphi(x,y,\xi) + t(p(y,\xi) - 1))} \, dt \, d\xi\\
		&= \lambda^n \int_{\R^n} \int_{-\infty}^\infty \hat \chi(t) q(t,x,y,\lambda \xi) \beta_1(p(y,\xi)) e^{i\lambda (\varphi(x,y,\xi) + t(p(y,\xi) - 1))} \, dt \, d\xi\\
		&\hspace{28em}+ O(\lambda^{-\infty})
\end{align*}
where $\beta_1$ is as before, that is with small support and with $\beta_1 \equiv 1$ near $1$. The $O(\lambda^{-\infty})$ bound on the discrepancy is uniform in $x$ and $y$, and is obtained by integration by parts in $t$. Hence,
\begin{align*}
	&\int_\Sigma \int_\Sigma \int_{-\infty}^\infty \hat \chi(t) e^{-it\lambda} B_\lambda e^{it\lap \tg} B_\lambda^*(x,y) \, dt \, d\sigma(x) \, d\sigma(y) \\
	&= \lambda^n \idotsint e^{i\langle x' - w, \eta \rangle} B_\lambda(x',w,\eta) e^{i(\varphi(w,z,\xi) + t(p(z,\xi) - 1))} \hat \chi(t) q(t,w,z,\xi) \\
	&\hspace{9em} \beta_1(|\xi|) e^{i\langle z - y', \zeta \rangle} \overline{B_\lambda(y',z,\zeta)} \, dt \, dx' \, dy' \, dw \, dz \, d\eta \, d\zeta \, d\xi + O(\lambda^{-\infty}).
\end{align*}
We perform the change of variables $\eta \mapsto \lambda \eta$ and $\zeta \mapsto \lambda \zeta$, and write $\xi = \xi' + r\omega$ in cylindrical coordinates with $r \in (0,\infty)$ and $\omega \in S^{n-d-1}$.  The integral on the right hand side is then
\begin{align}
	\nonumber = \lambda^{3n} \idotsint e^{i\lambda \Phi(t,x',y',\xi',r,\omega,w,z,\eta,\zeta)} a(\lambda; t,x',y',\xi',r, \omega ,w,z,\eta,\zeta) & \\
	\label{local prop integral} &\hspace{-8em} \, dt \, dx' \, dy' \, d\xi' \, dr \, d\omega \, dw \, dz \, d\eta \, d\zeta
\end{align}
where $d\omega$ denotes the standard volume measure on $S^{n-d-1}$,
\begin{align*}
	&\Phi(\lambda; t,x',y',\xi',r,\omega,w,z,\eta,\zeta)\\
	&\hspace{4em}= \langle x' - w, \eta \rangle + \varphi(w,z,\xi'+r\omega) + t(p(z,\xi' + r\omega) - 1) + \langle z - y' , \zeta \rangle,
\end{align*}
and
\begin{align*}
	&a(\lambda; t,x',y',\xi',r,\omega,w,z,\eta,\zeta)\\
	&= \hat \chi(t) b(x') \overline{b(y')} \beta_0(|x' - w|) \beta_0(|y' - z|)q(t,w,z,\lambda (\xi' + r\omega)) \\
	&\hspace{10em} \beta_1(|\eta|) \beta_1(|\zeta|) \beta_1(p(z,\xi' + r\omega)) a(\eta/|\eta|) a(\zeta/|\zeta|) r^{n-d-1}.
\end{align*}
Note all derivatives of $a$ are uniformly bounded for $\lambda \geq 1$.

We will use the method of stationary phase in variables $t$, $x'$, $\xi'$, $r$, $w$, $z$, $\eta$, and $\zeta$. Instead of doing so all at once with eight variables, we break it into two stages -- the first involving $w$, $z$, $\eta$, and $\zeta$, and the second involving the remaining four. We begin by fixing $x'$, $y'$, and $\xi$ and by performing stationary phase with respect to $w$, $z$, $\eta$, and $\zeta$. The gradient of the phase function in these variables is
\[
	\nabla_{w,z,\eta,\zeta} \Phi = \begin{bmatrix}
		-\eta + \xi + O(|w-z||\xi|) \\
		\zeta - \xi + O(|w-z||\xi|) \\
		x' - w \\
		y' - z
	\end{bmatrix}
\]
which, when $x' = y'$, has a critical point at $w = z = y'$ and $\eta = \zeta = \xi$. The Hessian matrix at this point is
\[
	\nabla^2_{w,z,\eta,\zeta} \Phi = \begin{bmatrix}
		* & * & -I & 0 \\
		* & * & 0 & I \\
		-I & 0 & 0 & 0 \\
		0 & I & 0 & 0
	\end{bmatrix}
\]
which has determinant $-1$ and signature $0$. By ~\cite[Corollary 1.1.8]{SFIO} and after perhaps restricting the support of $a$, the integral \eqref{local prop integral} is
\begin{align} \label{local prop integral 2}
	= \lambda^n \idotsint e^{i\lambda \Psi(t,x',y',\xi',r,\omega)} a(\lambda; t, x',y', \xi', r,\omega) \, dt \, dx' \, dy' \, d\xi' \, dr
\end{align}
with phase
\[
	\Psi(t,x',y',r,\omega) = \varphi(x',y',\xi) + t(p(y',\xi) - 1)
\]
and where the amplitude has compact support and has uniformly bounded derivatives in all variables for $\lambda \geq 1$. Next we fix $y'$ and $\omega$ and perform stationary phase in the remaining variables $t,r,x',$ and $\xi'$. We have
\[
	\nabla_{t,r,x',\xi'} \Psi = \begin{bmatrix}
		p(y',\xi) - 1 \\
		t\partial_r p(y,\xi) + O(|x' - y'|^2|\xi|) \\
		\xi' + O(|x' - y'||\xi|)\\
		x' - y' + t \nabla_{\xi'} p(y',\xi) + O(|x'-y'|^2)
	\end{bmatrix}
\]
which has a critical point at $(t,r,x',\xi') = (0,1,y',0)$ whereat we have the Hessian
\[
	\nabla^2_{t,r,x',\xi'} \Psi = \begin{bmatrix}
		0 & 1 & 0 & 0 \\
		1 & 0 & 0 & 0 \\
		0 & 0 & * & I \\
		0 & 0 & I & 0
	\end{bmatrix}
\]
where in the computations we use
\[
	p(y',\xi) = \sqrt{r^2 + \sum_{i,j = 0}^d g_\Sigma^{ij}(y')\xi_i' \xi_j'},
\]
a consequence of the construction of our Fermi coordinates \eqref{fermi local metric}. By using stationary phase ~\cite[Corollary 1.1.8]{SFIO} in $2d + 2$ variables, \eqref{local prop integral 2} is $O(\lambda^{n-d-1})$, as desired.
\end{proof}

\begin{proof}[Proof of Proposition \ref{remainder prop}.]
Let $\chi$ be as in the proof of Proposition \ref{intermediate prop}. It suffices to show
\[
	\sum_j \chi(\lambda_j - \lambda) |R_\lambda e_j(x')|^2 \leq C_N \lambda^{-N} \qquad N = 1,2,\ldots
\]
uniformly for $x \in \Sigma$. Using a similar reduction as before, the sum on the left is
\begin{align*}
	&\frac{1}{2\pi} \int_{-\infty}^\infty \hat \chi(t) e^{-it\lambda} R_\lambda e^{it\lap g} R^*_\lambda(x',x') \, dt.
\end{align*}
Using the argument in the proof of Proposition \ref{intermediate prop}, the expression above is
\begin{align*}
	= \lambda^{3n} \idotsint e^{i\lambda \Phi(t,x',w,z,\eta,\zeta,\xi)} a(\lambda; t,x',w,z,\eta,\zeta,\xi) \, dt \, dw \, dz \, d\eta \, d\zeta \, d\xi
\end{align*}
where
\[
	\Phi(t,x',w,z,\eta,\zeta,\xi) = \langle x' - w, \eta \rangle + \varphi(w,z,\xi) + t(p(z,\xi) - 1) + \langle z - x', \zeta \rangle
\]
and
\begin{align*}
	&a(\lambda; t,x',w,z,\eta,\zeta,\xi) = \hat \chi(t) |b(x')|^2 \beta_0(|x' - w|)\beta_0(|x' - z|) q(t,w,z,\lambda \xi)\\
	&\hspace{18em} (1 - \beta_1(|\eta|)) (1 - \beta_1(|\zeta|))\beta_1(p(z,\xi)).
\end{align*}
As before, the critical points of $\Phi$ occur only where $\eta = \zeta = \xi$. By the construction of our coordinates,
\[
	p(x',\xi) = (1 + O(|x'|^2))|\xi|
\]
and so we may adjust the support of $b$ so that $(1 - \beta_1(|\xi|))\beta_1(p(x',\xi)) \equiv 0$. Hence, the critical points of $\Phi$ lie outside the support of the amplitude and the desired bound follows from nonstationary phase ~\cite[Lemma 0.4.7]{SFIO}.
\end{proof}


\section{Kernel Bounds}

We require a characterization of the kernels $K_{\alpha}$ defined in \eqref{conjugated kernel def} to proceed. Note first that if $x$ and $y$ are expressed in our Fermi coordinates \eqref{fermi coordinates} about $\Sigma$, 
\begin{align}
	\label{kernel alpha def} &K_{\alpha}(T,\lambda;x,y)\\
	\nonumber &= \frac{1}{(2\pi)^{2n}} \iiiint e^{i\langle x - w, \eta \rangle} B_\lambda(x, w,\eta) K(T,\lambda;\alpha \tilde w,\tilde z) e^{i\langle z - y, \zeta \rangle} \overline{B_\lambda(y,z,\zeta)}\\
	\nonumber & \hspace{26em}\, dw \, dz \, d\eta \, d\zeta
\end{align}
where $\tilde w$ and $\tilde z$ are the respective lifts of $w$ and $z$ to the Dirichlet domain $D$ and
\begin{equation}\label{kernel def}
	K(T,\lambda; \tilde x, \tilde y) = \int (1 - \beta(t/R)) \hat \chi(t/T) e^{-it\lambda} \cos(t\lap \tg)(\tilde x,\tilde y) \, dt.
\end{equation}
We begin by developing a characterization of the kernel $K(T,\lambda;\tilde x,\tilde y)$ for $\tilde x, \tilde y \in \tilde M$ with $d_\tg(\tilde x,\tilde y)$ bounded away from zero as in ~\cite{Berard, CSPer, Gauss}. In what follows, we draw liberally from Sogge's text, \emph{Hangzhou Lectures on Eigenfunctions of the Laplacian} ~\cite{Hang}, for its arguments and notation, and also B\'erard's article ~\cite{Berard} for asymptotic bounds on derivatives of the distance function and the coefficients of the Hadamard parametrix.

\begin{lemma} \label{kernel lemma}
	Fix a positive integer $m$. There exist functions $a_\pm(T,\lambda;\tilde x,\tilde y)$ and $R(T,\lambda; \tilde x, \tilde y)$ depending on $m$ such that
	\[
		K(T,\lambda; \tilde x, \tilde y) = \lambda^\frac{n-1}{2} \sum_\pm a_\pm(T,\lambda;\tilde x,\tilde y) e^{\pm i \lambda d_\tg(\tilde x,\tilde y)} + R(T,\lambda;\tilde x,\tilde y)
	\]
	where if $d_\tg(\tilde x,\tilde y) \geq 1$,
	\begin{equation}\label{kernel a bounds}
		|\Delta_x^j \Delta_y^k a_\pm(T,\lambda;\tilde x,\tilde y)| \leq C_{j,k} e^{C_{j,k} d_\tg(\tilde x,\tilde y)} \qquad j,k = 0,1,2,\ldots
	\end{equation}
	and
	\begin{equation}\label{kernel R bounds}
		|R(T,\lambda;\tilde x,\tilde y)| \lesssim e^{CT}\lambda^{-m}.
	\end{equation}
	Moreover if $d_\tg(\tilde x,\tilde y) \leq R$,
	\begin{equation} \label{kernel lemma local}
		|K(T,\lambda;\tilde x,\tilde y)| \lesssim e^{C T} \lambda^{-m}.
	\end{equation}
\end{lemma}

\begin{proof}
By Theorem 2.4.1 and Remark 1.2.5 of ~\cite{Hang},
	\begin{equation} \label{parametrix 1}
	\cos(t\lap \tg)(\tilde x,\tilde y) = \sum_{\nu = 0}^N \alpha_\nu(\tilde x,\tilde y) \partial_t E_\nu(t,d_\tg(\tilde x,\tilde y)) + R_N(t,\tilde x,\tilde y)
\end{equation}
where $\partial_t E_\nu(t,r)$ is some distribution supported on $|t| \leq r$, and if $\tilde x$ is expressed in geodesic normal coordinates about $\tilde y$ with metric $\tg$, the coefficients $\alpha_\nu$ are defined inductively by
\[
	\alpha_0(\tilde x,\tilde y) = |\tg(\tilde x)|^{-1/4}
\]
and
\begin{equation} \label{alpha induction}
	\alpha_\nu(\tilde x,\tilde y) = \alpha_0(\tilde x,\tilde y) \int_0^1 t^{\nu-1} \frac{\Delta_\tg \alpha_{\nu - 1}(t\tilde x,\tilde y)}{\alpha_0(t\tilde x,\tilde y)} \, dt, \qquad \nu = 1,2,3,\ldots
\end{equation}
where here $\Delta_\tg$ operates in the $\tilde x$ variable. Note $\alpha_\nu$ are defined on all of $\tilde M$ since $|\tg(\tilde x)|$ is nonvanishing.
Finally the remainder term satisfies
\[
	(\partial_t^2 - \Delta_\tg)R_N(t,\tilde x,\tilde y) = \Delta_\tg \alpha_N(\tilde x,\tilde y) \partial_t E_N(t,d_\tg(\tilde x,\tilde y)).
\]
where $\Delta_\tg$ operates in the $\tilde x$ variable. In addition, the appendix of ~\cite{Berard} provides us with exponential bounds,
\[
	|\Delta_{\tilde y}^j \alpha_\nu(\tilde x,\tilde y)| \leq C_j e^{C_j d_\tg(\tilde x,\tilde y)} \qquad j = 0,1,2,\ldots,
\]
which, with the fact that $\cos(t\lap \tg)$ is self-adjoint, provide us with the same bounds on derivatives in $\tilde x$
\[
	|\Delta_{\tilde x}^j \alpha_\nu(x,y)| \leq C_j e^{C_j d_\tg(x,y)} \qquad j = 0,1,2,\ldots
\]
(see ~\cite{Gauss}). Proposition \ref{appendix prop} in the appendix provides us with exponential bounds on the mixed derivatives,
\begin{equation}\label{coefficient bounds}
	|\Delta_{\tilde x}^j \Delta_{\tilde y}^k \alpha_\nu(\tilde x,\tilde y)| \leq C_{j,k} e^{C_{j,k} d_\tg(\tilde x,\tilde y)} \qquad j,k = 0,1,2,\ldots.
\end{equation}
The same proposition and B\'erard's exponential bounds on derivatives of the distance function provide
\begin{equation}\label{distance bounds}
	|\Delta_{\tilde x}^j \Delta_{\tilde y}^k d_\tg(\tilde x,\tilde y)| \leq C_{j,k} e^{C_{j,k} d_\tg(\tilde x,\tilde y)} \qquad j,k = 0,1,2,\ldots.
\end{equation}
From \eqref{coefficient bounds}, \eqref{distance bounds}, an energy estimate argument in ~\cite[\S 3.1]{Hang}, and the fact that $\partial_t E_\nu(t,r)$ is supported on $|t| \leq r$, we have that $R_N$ is $C^m$ and satisfies bounds
\[
	|\partial_t^j R_N(t,\tilde x,\tilde y)| \leq C_j e^{C_j d_\tg(\tilde x,\tilde y)}|t|^{2N+2-n-j} \qquad \text{ for } j = 0,1,\ldots, m
\]
provided $N > m + \frac{n+1}{2}$. Integration by parts $m$ times yields the bound
\begin{equation} \label{kernel remainder 1}
	\left|\int_{\infty}^\infty (1 - \beta(t/R)) \hat \chi(t/T) e^{-it\lambda} R_N(t,\tilde x,\tilde y) \, dt \right| \lesssim e^{C_{N,m}T} \lambda^{-m}
\end{equation}
as desired by \eqref{kernel R bounds}.

In light of \eqref{coefficient bounds} and \eqref{distance bounds}, it suffices to show
\begin{equation}\label{kernel lemma suffice 1}
	\int_{-\infty}^\infty (1 - \beta(t/R)) \hat \chi(t/T) e^{-it\lambda} \partial_t E_\nu(t,r) \, dt = \lambda^\frac{n-1}{2} \sum_\pm a_\pm^\nu(T,\lambda;r) e^{\pm i \lambda r}
\end{equation}
modulo terms whose contributions can be absorbed by the remainder $R(T,\lambda; \tilde x, \tilde y)$,
where $a_\pm^\nu$ satisfy bounds
\begin{equation}\label{kernel lemma suffice 2}
	|\partial_r^\ell a_\pm^\nu(T,\lambda;r)| \leq C_{\nu,\ell} \lambda^{-\nu} P_{\ell,\nu,k,j}(r) \qquad \text{ for } \ell = 0,1,2,\ldots, \ T \geq 1, \ r \geq 1
\end{equation}
where $P_{\ell,\nu,k,j}$ is some polynomial.
By ~\cite[Remark 1.2.5]{Hang}, $\partial_t E_\nu(t,r)$ is a finite linear combination of distributions
\begin{equation}\label{def E nu}
	t^j \int_{|\xi| \geq 1} e^{ir\xi_1 \pm it|\xi|} |\xi|^{-\nu-k} \, d\xi \qquad \text{for } j + k = \nu, \ j,k = 0,1,2,\ldots
\end{equation}
modulo smooth terms whose derivatives grow at most polynomially in $t$ and $r$. The contribution of these discrepancy terms hence satisfy the same bounds as \eqref{kernel lemma local} and may be absorbed by the remainder. The contribution of each term \eqref{def E nu} to the integral in \eqref{kernel lemma suffice 1} is
\[
	\int_{|\xi| \geq 1} \int_{-\infty}^\infty t^j (1 - \beta(t/R)) \hat \chi(t/T) e^{-it\lambda} e^{ir\xi_1 \pm it|\xi|} |\xi|^{-\nu-k} \, dt \, d\xi.
\]
If the sign in the exponent is negative, the integral satisfies good bounds by integrating by parts in $t$ and may be absorbed into the remainder, so it suffices only to consider the situation where the sign in the exponent is positive. In this case, we perform a change of variables $\xi \mapsto \lambda \xi$ and obtain
\begin{align*}
	&\int_{|\xi| \geq 1} \int_{-\infty}^\infty t^j (1 - \beta(t/R)) \hat \chi(t/T) e^{i(r\xi_1 + t(|\xi| - \lambda))} |\xi|^{-\nu-k} \, dt \, d\xi \\
	&= \lambda^{n-\nu-k} \int_{|\xi| \geq \lambda^{-1}} \int_{-\infty}^\infty t^j (1 - \beta(t/R)) \hat \chi(t/T) e^{i\lambda (r\xi_1 + t(|\xi| - 1))} |\xi|^{-\nu-k} \, dt \, d\xi.
\end{align*}
Let $\beta_1 \in C_0^\infty(\R,[0,1])$ be equal to $1$ near $1$ and have small support. We cut the integral in the second line into $\beta_1(|\xi|)$ and $(1 - \beta_1(|\xi|))$ parts. The latter cut contributes a $O(T^{j-m+1} \lambda^{-m})$ term by integrating by parts in the $t$ variable $m$ times, and we let it be absorbed into the remainder. The $\beta_1(|\xi|)$ cut comes to
\[
	\lambda^{n-\nu-k} \int_{\R^n} \int_{-\infty}^\infty t^j (1 - \beta(t/R)) \hat \chi(t/T) e^{i\lambda (r\xi_1 + t(|\xi| - 1))} \beta_1(|\xi|)|\xi|^{-\nu-k} \, dt \, d\xi.
\]
We take a moment to note that the integrand is supported on $|t| \geq 2R$, and hence if $r \leq R$, the gradient in $\xi$ of the phase satisfies
\[
	|\nabla_\xi(r \xi_1 + t(|\xi| - 1))| = |r e_1 + t \xi/|\xi|| \geq R
\]
for all $t$ in the support of the integrand by the triangle inequality. Nonstationary phase and the bounds on our remainder term thus far yields \eqref{kernel lemma local}.

From now on, we take $r \geq R$. By a change of coordinates $t \mapsto rt$, we write the integral as
\begin{align*}
	&\lambda^{n-\nu-k} \int_{\R^n} \int_{-\infty}^\infty t^j (1 - \beta(t/R)) \hat \chi(t/T) e^{i\lambda (r\xi_1 + t(|\xi| - 1))} \beta_1(|\xi|)|\xi|^{-\nu-k} \, dt \, d\xi \\
	&= \lambda^{n-\nu-k} r^{j+1} \int_{\R^n} \int_{-\infty}^\infty t^j (1 - \beta(rt/R)) \hat \chi(rt/T) e^{i\lambda r(\xi_1 + t(|\xi| - 1))} \beta_1(|\xi|)|\xi|^{-\nu-k} \, dt \, d\xi.
\end{align*}
We cut the integral one last time into $\beta_1(|t|)$ and $(1 - \beta_1(|t|))$ components. By H\"uygen's principle, we only consider the situation where $r \leq T$, and hence $\beta_1(|t|)(1 - \beta(rt/R))\hat \chi(rt/T)$ and $(1 - \beta_1(|t|))(1 - \beta(rt/R))\hat \chi(rt/T)$ have bounded derivatives in $t$ and $r$ of all orders. The norm of the $\xi$-gradient of the phase function is
\[
	|\nabla_\xi (\xi_1 + t(|\xi| - 1))| = |e_1 + t\xi/|\xi||
\]
which is again bounded away from $0$ on the support of $(1 - \beta_1(|t|))$ and so contributes a term to be absorbed by the remainder by nonstationary phase. We write the $\beta_1(|t|)$ cut as $I_+(T,\lambda;r) + I_-(T,\lambda;r)$
where
\begin{align*}
	&I_\pm(T,\lambda;r)\\
	&= \lambda^{n-\nu-k} r^{j+1} \int_{\R^n} \int_{-\infty}^\infty t^j (1 - \beta(rt/R)) \beta_1(\pm t) \hat \chi(rt/T) e^{i\lambda r(\xi_1 + t(|\xi| - 1))} \beta_1(|\xi|)|\xi|^{-\nu-k} \, dt \, d\xi.
\end{align*}
The phase function of $I_\pm$ has a critical point at $(t,\xi) = \pm(1,-e_1)$ at which the Hessian of the phase function,
\[
	\pm \begin{bmatrix}
		0 & -1 & 0 \\
		-1 & 0 & 0 \\
		0 & 0 & I
	\end{bmatrix},
\]
is nondegenerate. Stationary phase ~\cite[Proposition 4.1.2]{Hang} yields
\[
	|\partial_r^\ell (e^{\pm i r \lambda} I_\pm(T,\lambda;r))| \leq C_{\ell,\nu,k,j} \lambda^{\frac{n-1}{2} - \nu - k} r^{j - \ell - \frac{n-1}{2}},
\]
from which \eqref{kernel lemma suffice 1} and \eqref{kernel lemma suffice 2} follow.
\end{proof}

Set
\[
	\Gamma_R = \left\{ \alpha \in \Gamma : \sup_{x,y \in \supp b} d_\tg(\alpha \tilde x, \tilde y) \leq R \right\}.
\]
The contribution of the terms of $\Gamma_R$ to the sum \eqref{reduction final} are $O(e^{CT}\lambda^{-m})$ by \eqref{kernel lemma local} of the lemma, which is better than we need. Moreover by restricting the support of $b$, we ensure that
\begin{equation}\label{def Gamma_R}
	\inf_{x,y\in \supp b} d_\tg(\alpha \tilde x, \tilde y) \geq R-1 \qquad \text{ if } \alpha \in \Gamma \setminus \Gamma_R.
\end{equation}
In light of this, what remains is to show that
\begin{equation}\label{reduction long time}
	\sum_{\Gamma \setminus \Gamma_R} \left| \int_\Sigma \int_\Sigma K_\alpha(T,\lambda;x,y) \, d\sigma(x) \, d\sigma(y) \right| \lesssim e^{CT} \lambda^\delta.
\end{equation}

The next lemma uses the previous to characterize the conjugated kernel $K_\alpha$. Here the function of the operators $B_\lambda$ begins to surface. Conjugating $K$ by $B_\lambda$ filters out points $\tilde x$ and $\tilde y$ in $\tilde M$ for which the geodesic connecting $\tilde y$ to $\alpha \tilde x$ departs and arrives in dissimilar directions. This will be very useful in Section 5, when we need to control the gradient of the phase function $d_\tg(\alpha \tilde x, \tilde y)$. As usual, $\tilde x$ and $\tilde y$ denote the respective lifts of $x$ and $y$ to the Dirichlet domain $D$.

\begin{lemma}\label{kernel alpha lemma} We have
\begin{align} \label{kernel alpha lemma 1}
	K_{\alpha}(T,\lambda; x,y) = \lambda^\frac{n-1}{2}\sum_\pm a_{\alpha, \pm}(T,\lambda;x,y) e^{\pm i \lambda d_\tg(\alpha \tilde x,\tilde y)} + O(e^{CT} \lambda^\delta)
\end{align}
where the amplitude $a_{\alpha,\pm}$ satisfies bounds
\begin{equation} \label{kernel alpha lemma 2}
	|\Delta_x^j \Delta_y^k a_{\alpha, \pm}(T,\lambda;x,y)| \leq C_{i,j} e^{C_{i,j} d_\tg(\alpha \tilde x, \tilde y)}
\end{equation}
and is supported on $\supp_x B \times \supp_x B$.
Moreover, there exists an open conical neighborhood $U \subset T^* \tilde M$ which can be made small by restricting the support of $B_\lambda$ such that 
\begin{equation} \label{kernel alpha lemma 3}
	|a_{\alpha,\pm}(T,\lambda;x,y)| \leq C_{U,N} e^{C_{U,N}d_\tg(\alpha \tilde x,\tilde y)} \lambda^{-N} \qquad N = 1,2,\ldots
\end{equation}
for all $x$ and $y$ for which neither of
\begin{align*}
	&(\gamma'(0), \alpha^*\gamma'(1)) \in U \times U \qquad \text{ nor }\\
	&(-\gamma'(0),-\alpha^*\gamma'(1)) \in U \times U
\end{align*}
hold,
where $\gamma$ is the constant-speed geodesic with $\gamma(0) = \tilde y$ and $\gamma(1) = \alpha \tilde x$, and where $\gamma'$ is understood as an element in $T^* \tilde M$, and where $\alpha^*$ is the pullback on the cotangent bundle through $\alpha$.
\end{lemma}

\begin{proof}
	By Lemma \ref{kernel lemma}, we have
\begin{align*}
	&K_{\alpha}(T,\lambda; x,y)\\
	&= \frac{\lambda^\frac{n-1}{2}}{(2\pi)^{2n}} \sum_\pm \iiiint e^{i\langle x - w, \eta \rangle} B_\lambda(x,w,\eta) a_\pm(T,\lambda;\alpha \tilde w,\tilde z)e^{\pm i\lambda d_\tg(\alpha \tilde w, \tilde z)}\\
	&\hspace{18em} e^{i\langle z - y, \zeta \rangle} \overline{ B_\lambda(y, z,\zeta)} \, d w \, d z \, d\eta \, d\zeta \\
	&\hspace{2em}+ \frac{1}{(2\pi)^{2n}} \iiiint e^{i\langle x - w, \eta \rangle} B_\lambda( x, w,\eta) R(T,\lambda;\alpha \tilde w,\tilde z)e^{i\langle z - y, \zeta \rangle} \overline{ B_\lambda( y, z,\zeta)}\\
	&\hspace{28em}\, dw \, d z \, d\eta \, d\zeta.
\end{align*}
The second integral on the right hand side is $O(e^{CT} \lambda^{\delta})$ by taking $m$ in \eqref{kernel R bounds} greater than $2n - \delta$ and the fact that
\[
	\int_{\R^n} \int_{\R^n} | B_\lambda(x,w,\eta) | \, dw \, d\eta = O(\lambda^{n}).
\]
It suffices then to equate the first term to the right hand side of \eqref{kernel alpha lemma 1}. Using a change of variables $\eta \mapsto \lambda \eta$ and $\zeta \mapsto \lambda \zeta$, this is
\[
	\frac{\lambda^{2n + \frac{n-1}{2}}}{(2\pi)^{2n}} \sum_\pm \iiiint e^{i\lambda \Phi_\pm(x,y,w,z,\eta,\zeta)} A(T,\lambda;x,y,w,z,\eta,\zeta) \, dw \, dz \, d\eta \, d\zeta
\]
where
\[
	\Phi_\pm(x,y,w,z,\eta,\zeta) = \langle x - w, \eta \rangle \pm d_\tg(\alpha \tilde w, \tilde z) + \langle z - y, \zeta \rangle
\]
and by \eqref{B symbol},
\begin{multline}
\label{kernel alpha lemma A}
A(T,\lambda;x,y,w,z,\eta,\zeta) = \beta_0(|x - w|) \beta_0(|z - y|) \beta_0(|x^\perp|) \beta_0(|y^\perp|) b(x') \overline{b(y')} \\
a_\pm(T,\lambda; \alpha \tilde w, \tilde z)a(\eta/|\eta|) a(\zeta/|\zeta|) \beta_1(|\eta|) \beta_1(|\zeta|)
\end{multline}
For clarity, we focus only on the $\Phi_+$ component; the argument for the alternate sign is the same. The Euclidean gradient of the phase function with respect to the variables of integration is
\[
	\nabla_{w,z,\eta,\zeta} \Phi_\pm = \begin{bmatrix}
		-\eta + \nabla_{\tilde w} d_\tg(\alpha \tilde w,\tilde z) \\
		\zeta + \nabla_{\tilde z} d_\tg(\alpha \tilde w,\tilde z) \\
		x - w \\
		z - y
	\end{bmatrix}
\]
which has a critical point at $(w,z,\eta,\zeta) = (x,y, \nabla_{\tilde x} d_\tg(\alpha \tilde x, \tilde y), - \nabla_{\tilde y} d_\tg(\alpha \tilde x, \tilde y))$ at which the phase takes the value $d_\tg(\alpha \tilde x, \tilde y)$ and has Hessian
\[
	\nabla^2_{w,z,\eta,\zeta} \Phi_\pm = \begin{bmatrix}
		* & * & -I & 0 \\
		* & * & 0 & I \\
		-I & 0 & 0 & 0 \\
		0 & I & 0 & 0
	\end{bmatrix},
\]
which has determinant $-1$. We have \eqref{kernel alpha lemma 1} and \eqref{kernel alpha lemma 2} by \eqref{distance bounds}, \eqref{kernel a bounds}, and ~\cite[Corollary 1.1.8]{SFIO}. Now assume without loss of generality $U$ is an open conic neighborhood in $\R^n$ whose projection onto the manifold contains the support of $a$ in \eqref{kernel alpha lemma A}. If $\nabla_{\tilde x} d_\tg(\alpha \tilde x, \tilde y)$ lies in the complement of $U$, then
\[
	|-\eta + \nabla_{\tilde x} d_\tg(\alpha \tilde x, \tilde y)| \geq c > 0
\]
on the support of $A$ for some constant $c$ depending on $U$. Hence,
\[
	|-\eta + \nabla_{\tilde w} d_\tg(\alpha \tilde w, \tilde z)| \geq c - |\nabla_{\tilde x} d_\tg(\alpha \tilde x, \tilde y) - \nabla_{\tilde w} d_\tg(\alpha \tilde w, \tilde z)|.
\]
In the next section, we will show that the Hessian of the distance function is uniformly bounded on the entirety of $\tilde M \times \tilde M$ minus a neighborhood of the diagonal (see Remark \ref{bounded hessian remark}). Moreover since $\tilde x, \tilde y, \tilde w,$ and $\tilde z$ are all in the same local coordinates, the Christoffel symbols of the metric are bounded. Hence, the Euclidean Hessian of $d_\tg(\alpha \tilde x, \tilde y)$ in both variables is uniformly bounded\footnote{See \eqref{hessian in coordinates} for the relationship between the Hessian on a manifold and the Euclidean Hessian in local coordinates.} in $\alpha$ and
\[
	|\nabla_{\tilde x} d_\tg(\alpha \tilde x, \tilde y) - \nabla_{\tilde w} d_\tg(\alpha \tilde w, \tilde z)| \leq C(|x - w| + |y - z|)
\]
by the mean value theorem. We restrict the support of $\beta_0$ in \eqref{kernel alpha lemma A} so that $|-\eta + \nabla_{\tilde w} d_\tg(\alpha \tilde w, \tilde z)|$ is bounded away from $0$ uniformly in $\alpha$. We remark that the covector $\langle \ \cdot \ , \nabla_{\tilde x} d_\tg(\alpha \tilde x, \tilde y) \rangle$ with the Euclidean inner product is precisely the dual of $\gamma'(1)/|\gamma'(1)|$ pulled back by $\alpha$. The desired bound \eqref{kernel alpha lemma 3} then follows from \eqref{distance bounds}, \eqref{kernel a bounds}, and nonstationary phase ~\cite[Lemma 0.4.7]{SFIO} in the $w$ variable. The argument is similar if $-\nabla_{\tilde y} d_\tg(\alpha \tilde x, \tilde y)$ is in the complement of $U$.
\end{proof}

Let $\Gamma_{U}$ denote the subset of $\Gamma$ for which there exist $x$ and $y$ in the support of $a_{\alpha, \pm}$ such that the geodesic $\gamma : [0,1] \to \tilde M$ with $\gamma(0) = \tilde y$ and $\gamma(1) = \alpha \tilde x$ has both $\gamma'(0) \in U$ and $\alpha^* \gamma'(1) \in U$. Lemma \ref{kernel alpha lemma} and \eqref{finite sum} show us
\[
	\sum_{\alpha \in (\Gamma \setminus \Gamma_U) \setminus \Gamma_R} \left| \int_{ \Sigma} \int_{\Sigma} K_\alpha(T,\lambda;x,y) \, d\sigma(x) \, d\sigma(y) \right| \lesssim e^{CT}\lambda^{-m}
\]
for some $m$ which can be made large. So, \eqref{reduction long time} would follow from
\begin{equation}\label{reduction U}
	\sum_{\alpha \in \Gamma_U \setminus \Gamma_R} \left| \int_{\Sigma} \int_{\Sigma} K_\alpha(T,\lambda;x,y) \, d\sigma(x) \, d\sigma(y) \right| \lesssim e^{CT}\lambda^{\delta}.
\end{equation}

It is now time to specify the statements we require to prove Theorems \ref{main 1} and \ref{main 2}. Recall from \eqref{reduction 1} that the only requirement for the exponent $\delta$ is that it is less than $n - d - 1$. Propositions \ref{main 1 prop} and \ref{main 2 prop} along with Lemma \ref{kernel alpha lemma} and \eqref{finite sum} imply \eqref{reduction U} under the hypotheses of Theorem \ref{main 1} and Theorem \ref{main 2}, respectively.

\begin{proposition}\label{main 1 prop} Under the hypotheses of Theorem \ref{main 1}, we have
\[
	\left| \int_\Sigma \int_\Sigma a_{\alpha,\pm}(T,\lambda;x,y)e^{\pm i \lambda d_\tg(\alpha \tilde x, \tilde y)} \, d\sigma(x) \, d\sigma(y) \right| \lesssim e^{CT} \lambda^{-d/2} \qquad \text{ for } \alpha \in \Gamma_U \setminus \Gamma_R,
\]
where the constant $C$ is uniform in $\alpha$.
\end{proposition}

\begin{proposition}\label{main 2 prop}
	Assume the hypotheses of Theorem \ref{main 2}. If $\alpha \in \Gamma_U \setminus \Gamma_{R}$,
	\[
		\left| \int_\Sigma \int_\Sigma a_{\alpha,\pm}(T,\lambda;x,y)e^{\pm i \lambda d_\tg(\alpha \tilde x, \tilde y)} \, d\sigma(x) \, d\sigma(y) \right| \lesssim e^{CT} \lambda^{-n/2}
	\]
	where the constant $C$ is uniform in $\alpha$.
\end{proposition}


\section{Geometry and Phase Function Bounds} \label{GEOMETRY}

We will need some information about the first and second derivatives of the phase functions in Propositions \ref{main 1 prop} and \ref{main 2 prop}. This section will provide the tools necessary to do so. Specifically, we will compute the Hessian of the phase function using the second fundamental form of $\Sigma$ and of spheres in $\tilde M$. We then we verify Definition \ref{limiting curvature} and prove some useful properties of the second fundamental form of circles of large radius. Finally, we use these properties to provide good bounds on the Hessian of the phase function.
DoCarmo's text ~\cite{doCarmo} is our primary reference for this section.

We outline some basic facts before we begin.
For a general Riemannian manifold $(M,g)$ with Levi-Civita connection $\nabla$, the Hessian of $f \in C^\infty(M)$ is the quadratic form
\begin{equation}\label{def hessian}
	\Hess f(X,Y) = X(Y f) - (\nabla_X Y) f
\end{equation}
where $X$ and $Y$ are vector fields on $M$. For future use we note, in local coordinates $x = (x_1,\ldots, x_n)$,
\begin{equation}\label{hessian in coordinates}
	\Hess f \left(\frac{\partial}{\partial x_i},\frac{\partial}{\partial x_j} \right) = \frac{\partial^2 f}{\partial x_i \partial x_j} - \sum_k \Gamma_{ij}^k \frac{\partial f}{\partial x_k}
\end{equation}
and so if the Christoffel symbols $\Gamma_{ij}^k$ are small, the Hessian of $f$ is nearly the Euclidean Hessian. Suppose $\Sigma$ is a submanifold of $M$ with the induced metric $\overline g$ and Levi-Civita connection $\overline \nabla$. By \eqref{def hessian},
\begin{equation}\label{hessian restriction}
	\Hess_\Sigma f(X,Y) = \Hess_M f(X,Y) + \II_\Sigma(X,Y)f
\end{equation}
where $X,Y$ are vectors in $\Sigma$ and where $\II_\Sigma$ is the second fundamental form of $\Sigma$ in $M$, given by
\begin{equation} \label{def second fundamental form}
	\II_\Sigma(X,Y) = \nabla_X Y - \overline \nabla_X Y = (\nabla_X Y)^\perp,
\end{equation}
the orthogonal projection of $\nabla_X Y$ onto the normal bundle $N\Sigma$. The Hessians and the second fundamental form are tensorial and only depend on the value of $X$ and $Y$ at a point. (For details see ~\cite[Section 6.2]{doCarmo}.)

\subsection{Computing the Hessian of the Phase Function}

We will want to compute the Hessian of the phase functions from Propositions \ref{main 1 prop} and \ref{main 2 prop}, that is the function $\phi : \Sigma \times \Sigma \to \R$ given by
\[
	\phi(x,y) = d_\tg(\alpha \tilde x, \tilde y)
\]
where $\Sigma \times \Sigma$ is endowed with the product metric, where $\tilde x$ and $\tilde y$ are the respective lifts of $x$ and $y$ to our Dirichlet domain $D$ in the universal cover, and where $\alpha$ is a fixed, non-identity deck transformation. By \eqref{hessian restriction},
\begin{equation}\label{hessian sum}
	\Hess_{\Sigma \times \Sigma} \phi (X,Y) = \Hess_{\alpha \tilde \Sigma \times \tilde \Sigma} d_\tg (X, Y) = \Hess_{\tilde M \times \tilde M} d_\tg(X,Y) + \II_{\alpha \tilde \Sigma \times \tilde \Sigma}(X, Y) d_\tg
\end{equation}
where $X$ and $Y$ are both vectors in $\Sigma \times \Sigma$ with the same base point, but are also understood to be their respective lifts to $\alpha \tilde \Sigma \times \tilde \Sigma$ where appropriate. To compute the Hessian of the phase function, it suffices to compute the Hessian of $d_\tg$ on $\tilde M \times \tilde M$ and the second fundamental form of $\alpha \tilde \Sigma \times \tilde \Sigma$. To this end, we write
\[
	X = X_1 \oplus X_2 \qquad \text{ and } \qquad Y = Y_1 \oplus Y_2
\]
where $X_1$ and $Y_1$ are vectors on $\alpha \tilde \Sigma$ and $X_2$ and $Y_2$ are vectors on $\tilde \Sigma$ and write
\begin{align}
	\label{hessian decomposition} \Hess_{\tilde M \times \tilde M} d_\tg(X,Y) &= \sum_{i,j = 1,2} \Hess_{\tilde M \times \tilde M} d_\tg(X_i,Y_j) \qquad \text{ and } \\
	\label{second fundamental form decomposition} \II_{\alpha \tilde \Sigma \times \tilde \Sigma}(X, Y)d_\tg &= \sum_{i,j = 1,2} \II_{\alpha \tilde \Sigma \times \tilde \Sigma}(X_i, Y_j)d_\tg.
\end{align}
Note the $i \neq j$ terms of \eqref{second fundamental form decomposition} vanish and we are left with
\begin{equation}\label{second fundamental form decomposition'}
	\II_{\alpha \tilde \Sigma \times \tilde \Sigma}(X, Y) = \II_{\alpha \tilde \Sigma}(X_1, Y_1) d_\tg + \II_{\tilde \Sigma}(X_2, Y_2) d_\tg.
\end{equation}
The next lemma helps us compute the terms in \eqref{hessian decomposition}.

\begin{lemma} \label{hessian computation}
	Assume the notation of \eqref{hessian decomposition}, suppose $\tilde x$ and $\tilde y$ are \emph{any} points in $\tilde M$, let $r = d_\tg(\tilde x, \tilde y)$, and let $X_1,Y_1 \in T_{\tilde x} \tilde M$ and $X_2,Y_2 \in T_{\tilde y} \tilde M$. The following are true.
	\begin{enumerate}
		\item $X_1 d_\tg = \cos \theta$ where $\theta$ is the angle between $X_1$ and the first derivative of the geodesic adjoining $\tilde y$ to $\tilde x$. In particular, $X_1 d_\tg = 0$ if and only if $X_1$ is perpendicular to this geodesic. This holds similarly for $X_2 d_\tg$.
		\item We have absolute bounds
		\begin{align*}
			&|\Hess_{\tilde M \times \tilde M}d_\tg(X_1,Y_2)| \leq 2|X_1||Y_2|/r \qquad \text{ and } \\ &|\Hess_{\tilde M \times \tilde M}d_\tg(X_2,Y_1)| \leq 2|X_2||Y_1|/r.
		\end{align*}
		\item Let $S_{\tilde y}(r)$ denote the sphere in $\tilde M$ with center $\tilde y$ and radius $r$. Then,
		\begin{align*}
			\Hess_{\tilde M \times \tilde M}d_\tg(X_1, Y_1) &= -\II_{S_{\tilde y}(r)}(X_1',Y_1') d_\tg
		\end{align*}
		where $X_1'$ and $Y_1'$ are the orthogonal projections of $X_1$ and $Y_1$ onto $T_{\tilde x} S_{\tilde y}(r)$, respectively.
		We similarly have
		\[
			\Hess_{\tilde M \times \tilde M}d_\tg(X_2, Y_2) = -\II_{S_{\tilde x}(r)}(X_2',Y_2') d_\tg.
		\]
	\end{enumerate}
\end{lemma}

\begin{proof}
Fix $X_1$ and $Y_2$ as above and let $\sigma_1, \sigma_2 : (-\epsilon,\epsilon) \to \tilde M$ be curves with
\[
	\sigma_1'(0) = X_1 \in T_{\tilde x}\tilde M \qquad \text{ and } \qquad \sigma_2'(0) = Y_2 \in T_{\tilde y}\tilde M.
\]
We then define a map
\begin{align*}
	\gamma : (-\epsilon,\epsilon) \times (-\epsilon,\epsilon) \times [0,1] &\to \tilde M
\end{align*}
such that for all $u,v \in (-\epsilon,\epsilon)$,
\[
	\gamma(u,v,1) = \sigma_1(u) \qquad \text{ and } \qquad \gamma(u,v,0) = \sigma_2(v),
\]
and where $t \mapsto \gamma(u,v,t)$ traces out the constant-speed geodesic connecting $\sigma_2(v)$ to $\sigma_1(u)$. Since $\partial_u, \partial_v,$ and $\partial_t$ are coordinate vector fields in the domain of $\gamma$, the Lie brackets
\[
	[\partial_u, \partial_t] = 0, \qquad [\partial_v, \partial_t] = 0, \qquad \text{ and } \qquad [\partial_u, \partial_v] = 0
\]
all vanish. Hence,
\[
	0 = [\partial_u, \partial_t] \gamma = [\partial_u \gamma, \partial_v \gamma] = \nabla_{u} \partial_t \gamma - \nabla_{t} \partial_u \gamma,
\]
where $\nabla$ is the Levi-Civita connection on $\tilde M$ and where $\nabla_u$ and $\nabla_t$ are shorthand for the covariant derivative with respect to the vector fields $\partial_u \gamma$ and $\partial_t \gamma$. This and similar calculations yield the identities
\[
	\nabla_u \partial_t \gamma = \nabla_t \partial_u \gamma, \qquad \nabla_v \partial_t \gamma = \nabla_t \partial_v \gamma, \qquad \text{ and } \qquad \nabla_u \partial_v \gamma = \nabla_v \partial_u \gamma
\]
which we will use repeatedly and without reference. Next, we write
\[
	d_\tg(\sigma_1(u), \sigma_2(v))^2 = \int_0^1 |\partial_t \gamma(u,v,t)|^2 \, dt.
\]
Taking a derivative in $u$ of $\frac{1}{2} d_\tg^2$ yields
\begin{align*}
	d_\tg \partial_u d_\tg &= \int_0^1 \langle \partial_t \gamma(u,v,t), \nabla_u \partial_t \gamma(u,v,t) \rangle \, dt \\
	&= \int_0^1 \langle \partial_t \gamma(u,v,t), \nabla_t \partial_u \gamma(u,v,t) \rangle \, dt\\
	&= \int_0^1 \partial_t \langle \partial_t \gamma(u,v,t), \partial_u \gamma(u,v,t) \rangle \, dt \\
	&= \langle \partial_t \gamma(u,v,1), \partial_u \gamma(u,v,1) \rangle
\end{align*}
where the third line is due to the geodesic equation $\nabla_t \partial_t \gamma = 0$ and the fourth to the fundamental theorem of calculus. We deduce part (1) of the lemma from this and a similar computation in the other variable. Next, we take a derivative in $v$ and obtain
\begin{align*}
	d_\tg \partial_u \partial_v d_\tg &+ \partial_u d_\tg \partial_v d_\tg\\
	&= \langle \nabla_v \partial_t \gamma(u,v,1), \partial_u \gamma(u,v,1) \rangle + \langle \partial_t \gamma(u,v,1), \nabla_v \partial_u \gamma(u,v,1) \rangle.
\end{align*}
Note $\nabla_v \partial_u \gamma(u,v,1) = \nabla_u \partial_v \gamma(u,v,1) = 0$, since $\gamma(u,v,1)$ is constant in $v$. Hence,
\begin{equation} \label{h lemma 1}
	d_\tg \partial_u \partial_v d_\tg + \partial_u d_\tg \partial_v d_\tg = \langle \nabla_t \partial_v \gamma(u,v,1), X_1 \rangle.
\end{equation}
We pause here to make a couple observations. First, $t \mapsto \partial_v \gamma(0,0,t)$ is a Jacobi field along $t \mapsto \gamma(0,0,t)$ with boundary data
\[
	\partial_v \gamma(0,0,0) = Y_2 \qquad \text{ and } \qquad \partial_v \gamma(0,0,1) = 0.
\]
Observe that $\partial_u \partial_v d_\tg$ is independent of our choice of curves $\sigma_1$ and $\sigma_2$, and that
\[
	\Hess_{\tilde M \times \tilde M} d_\tg(X_1,Y_2) = X_1 (Y_2 d_\tg) = \partial_u \partial_v d_\tg(\sigma_1(u),\sigma_2(v))
\]
at $u = v = 0$. To get part (2) of the lemma, it suffices to show that the right side of \eqref{h lemma 1} is bounded by $2|X_1||Y_2|/d_\tg(\tilde x, \tilde y)$. Let $h(t)$ denote the inner product of $\partial_v \gamma(0,0,t)$ with the parallel translate of $\pm X_1$ along $\gamma$, with the sign chosen so that $h(0) \geq 0$. By the Jacobi equation,
\[
	h''(t) + R(t) h(t) = 0
\]
for some nonpositive function $R(t)$ depending on the Riemann curvature tensor. We may as well assert that $h$ be nontrivial and hence vanishes only at $1$. Then, $h \geq 0$ on $[0,1]$ and so
\[
	h''(t) \geq 0 \qquad \text{ for } t \in [0,1].
\]
By convexity,
\[
	0 \leq h(t) \leq h(0)(1 - t),
\]
and hence
\[
	0 \geq h'(1) \geq -h(0).
\]
We know $h'(1)$ is equal to the right hand side of \eqref{h lemma 1} up to a sign, and that $|h(0)| \leq |X_1||Y_2|$ by Cauchy-Schwarz. Furthermore, $|X_1 d_\tg| \leq |X_1|$ and $|Y_2 d_\tg| \leq |Y_2|$ by the triangle inequality. Hence,
\[
	|\partial_u \partial_v d_\tg| = \frac{| \langle \nabla_t \partial_v \gamma(0,0,1), X_1 \rangle - (\partial_u d_\tg) (\partial_v d_\tg) |}{d_\tg} \leq \frac{2 |X_1||Y_2|}{d_\tg},
\]
as desired.

Finally we prove part (3) of the lemma. Consider geodesic normal coordinates $(x_2,\ldots, x_n)$ at $\tilde x$ of the sphere $S_{\tilde y}(r)$. We take an extension $(x_1,x_2,\ldots, x_n)$ of these coordinates to a neighborhood of $\tilde M$, where $x_1$ is the radial coordinate. By the geodesic equation $\nabla_1 \partial_1 = 0$,
\[
	\Hess_{\tilde M}d_\tg(\partial_1,\partial_1) = \partial_1 (\partial_1 x_1) - (\nabla_1 \partial_1) x_1 = 0.
\]
Moreover if $i \neq 1$,
\[
	\Hess_{\tilde M}d_\tg(\partial_i,\partial_1) = \partial_i (\partial_1 x_1) - (\nabla_i \partial_1) x_1 = -\nabla_1 \partial_i x_1,
\]
where $\nabla_i \partial_1 = \nabla_1 \partial_i$ by a similar argument as in the proof of part (1). Notice that $\partial_i$ is a perpendicular Jacobi field along the $x_1$ coordinate geodesic. Hence, $\nabla_1 \partial_i$ is also perpendicular to the $x_1$ coordinate geodesic, and $\nabla_1 \partial_i x_1 = 0$. Then,
\[
	\Hess_{\tilde M}d_\tg(X_1,Y_1) = \Hess_{\tilde M}d_\tg(X_1', Y_1')
\]
where $X_1'$ and $Y_1'$ are the orthogonal projections of $X_1$ and $Y_1$ onto $T_{\tilde x} S_{\tilde y}(r)$. It suffices then to show 
\[
	\Hess_{\tilde M}d_\tg(X_1,Y_1) = -\II_{S_{\tilde y}(r)}(X_1,Y_1)d_\tg
\]
in the situation where $X_1$ and $Y_1$ are vectors tangent to the sphere $S_{\tilde y}(r)$. In this situation we have $Y_1 d_\tg \equiv 0$, whence
\[
	\Hess_{\tilde M}d_\tg(X_1,Y_1) = - (\nabla_{X_1} Y_1) d_\tg = - (\nabla_{X_1}Y_1)^\perp d_\tg = -\II_{S_{\tilde y}(r)}(X_1,Y_1) d_\tg,
\]
as desired.
\end{proof}

\begin{remark}\label{bounded hessian remark}
	By comparison with the Euclidean case, the Hessian of the distance function $d_\tg$ in one variable is uniformly bounded for $d_\tg \geq 1$ (see ~\cite[Theorem 1.1]{redbook}). This, part (2) of Lemma \ref{hessian computation}, and \eqref{hessian decomposition} show that the Hessian of $d_\tg$ in both variables is uniformly bounded for $d_\tg \geq 1$.
\end{remark}

Lemma \ref{hessian computation}, \eqref{hessian sum}, and \eqref{second fundamental form decomposition'} combined provide us with the crucial computation
\begin{align}
	\label{hessian computation final} \Hess_{\Sigma \times \Sigma} \phi(X,Y) &= \II_{\alpha \tilde \Sigma}(X_1,Y_1)d_\tg - \II_{S_{\tilde y}(d_\tg)}(X'_1,Y'_1)d_\tg \\
	\nonumber & + \II_{\tilde \Sigma}(X_2,Y_2)d_\tg - \II_{S_{\alpha \tilde x}(d_\tg)}(X_2',Y_2')d_\tg + R(X,Y)
\end{align}
where
\[
	|R(X,Y)| \leq 2(|X_1||Y_2| + |X_2||Y_1|)/d_\tg.
\]

\subsection{The Second Fundamental Form of Spheres}

To provide any useful bounds on $\Hess_{\Sigma \times \Sigma} \phi$, we need to understand the behavior of the second fundamental form of spheres of large radius. To do this, we first need to understand the behavior of the second fundamental form of spheres of \emph{infinite} radius -- horospheres. We begin by validating Definition \ref{limiting curvature} and providing some useful facts about $\II_{H(\tilde v)}$.

\begin{proposition} \label{limiting curvature prop}
	Let $v$ be any vector in the unit sphere bundle $SM$ and let $\II_{H(v)}$ be as in Definition \ref{limiting curvature}. Let $X$ and $Y$ denote vectors in $TM$ which share the same root with $v$ and are perpendicular to $v$. The following are true.
	\begin{enumerate}
		\item $\II_{H(v)}(X,Y)$ is well defined and bilinear in $X$ and $Y$.
		\item $\langle \II_{H(v)}(X,X), v \rangle \geq 0$ for all $X$.
		\item $\II_{H(v)}$ is continuous in $v$. More precisely, if $v$ is allowed to vary on any small open subset of $SM$ and $X(v)$ and $Y(v)$ depend continuously on $v$, then
		\[
			v \mapsto \langle \II_{H(v)}(X(v),Y(v)), v \rangle
		\]
		is a continuous function on $SM$.
	\end{enumerate}
\end{proposition}

The proof of Proposition is very similar to, and can actually be deduced from, the corresponding proposition in ~\cite{emmett1}. We provide a proof here for the sake of completeness.

\begin{proof}
Let $\gamma : \R \to M$ denote the geodesic with $\gamma'(0) = v$, and let $X$ be a vector perpendicular to $\gamma'(0)$. The definition requires we show there exists a unique Jacobi field $J$ along $\gamma$ so that $J(0) = X$ and
\begin{equation} \label{J bounded}
	J(r) = O(1) \qquad \text{ for } r \geq 0.
\end{equation}
The difference of any two such Jacobi fields satisfies \eqref{J bounded} and vanishes at $0$. By comparison with the Euclidean setting, the difference must have vanishing first derivative at $0$ as well, or else contradict \eqref{J bounded}. Hence $J$ is unique.

We make some simplifying reductions before proceeding with the proof of existence. We assume without loss of generality that $X$ has norm $1$ and extend $X$ by parallel transport to a vector field $X(r)$ along $\gamma$. Then if $h$ is a smooth function on $\R$ satisfying
\begin{equation}\label{h Jacobi}
	h'' + K h = 0
\end{equation}
where $K(r) = K(\gamma'(r), X(r)) \leq 0$ is the sectional curvature on $M$, then $hX$ is a Jacobi field along $\gamma$. Hence it suffices to construct an $h$ satisfying \eqref{h Jacobi} with
\begin{equation}\label{h Jacobi 0}
	h(0) = 1
\end{equation}
and
\begin{equation} \label{h Jacobi bounded}
	h(r) = O(1) \qquad \text{ for } r \geq 1.
\end{equation}
Let $h_s$ denote the unique function satisfying \eqref{h Jacobi}, \eqref{h Jacobi 0}, and $h_s(s) = 0$. We will show the limit
\[
	h = \lim_{s \to \infty} h_s = h_1 + \int_1^\infty \frac{\partial}{\partial s} h_s \, ds
\]
converges uniformly on compact sets and satisfies \eqref{h Jacobi bounded} (we obtain \eqref{h Jacobi} for free by uniform convergence on compact sets). Both this and \eqref{h Jacobi bounded} follow provided we show
\begin{equation} \label{convexity bounds}
	\left| \frac{\partial}{\partial s} h_s(r) \right| \leq \frac{r}{s^2} \qquad \text{ for } r \leq s.
\end{equation}
First note $\frac{\partial}{\partial s} h_s$ satisfies \eqref{h Jacobi}, that
\[
	\frac{\partial}{\partial s} h_s(0) = 0,
\]
and that
\[
	\frac{\partial}{\partial s} h_s(s) = -h_s'(s),
\]
the last line following from $h_s(s) = 0$ and chain rule. Since $K \leq 0$, the only function satisfying \eqref{h Jacobi} vanishing more than once is identically zero. Hence, $h_s \geq 0$ on $[0,s]$ and by convexity,
\[
	0 \leq  h_s(r) \leq 1 - \frac rs \qquad \text{ for } 0 \leq r \leq s.
\]
This and the limit definition of the derivative yields
\[
	0 \leq -h_s'(s) = \frac{\partial}{\partial s} h_s(s) \leq r/s.
\]
A similar convexity argument applied to $\frac{\partial}{\partial s} h_s$ yields \eqref{convexity bounds}. Setting $J = h X$ provides existence for (1) as noted before. Moreover, we have
\[
	\langle \II_{H(v)}(X,X) , \gamma'(0) \rangle = \left\langle - \frac{D}{dr} J(0), X(0) \right\rangle = -h'(0) \geq 0
\]
since $h'(0) > 0$ contradicts \eqref{h Jacobi bounded} by comparison with the flat case. Hence, we have (2).

Let $t \mapsto v(t)$ be a continuous function from a small neighborhood of $0\in \R$ to $SM$ and likewise index our geodesic $\gamma(t,r)$, parallel unit normal vector field $X(t,r)$, and functions $h(t,r)$ and $h_s(t,r)$. For each $r_0 > 0$, we will show
\begin{equation} \label{h continuous}
	\lim_{t \to 0} h(t,r) = h(0,r) \qquad \text{ uniformly for } r \in [0,r_0].
\end{equation}
It follows by \eqref{h Jacobi} and uniform convergence that
\[
	\lim_{t \to 0} \frac{\partial}{\partial r} h(t,0) = \frac{\partial}{\partial r} h(0,0)
\]
from which follows (3). Fix $1 > \epsilon > 0$ and let $s = 3r_0/\epsilon$. We have
\begin{align*}
	&|h(t,r) - h(0,r)|\\
	&\leq |h(t,r) - h_{s}(t,r)| + |h_{s}(t,r) - h_{s}(0,r)| + |h_{s}(0,r) - h(0,r)|.
\end{align*}
Using \eqref{convexity bounds} and integrating over $[s,\infty)$, the first and last terms are both bounded above by $\epsilon/3$. We can make the middle term less than $\epsilon/3$ by noting $h_s(t,r)$ is uniformly continuous in $t$ for $r \in [0,r_0]$ and taking $|t|$ sufficiently small. Hence, $|h(t,r) - h(0,r)| < \epsilon$ and we have \eqref{h continuous}.
\end{proof}

Let $v \in SM$ and $\gamma$ be a geodesic with $\gamma'(0) = v$ as in Definition \ref{limiting curvature}. Note $\II_{H(v)}$ only ever depends on the sectional curvature of $M$ along $\gamma(r)$ for $r \geq 0$. In particular if $X$ and $Y$ are vectors perpendicular to $v$ and $\tilde X$, $\tilde Y$, and $\tilde v$ are their respective lifts to the universal cover, $\II_{H(\gamma'(0))}(X,Y)$ lifts to $\II_{H(\tilde v)}(\tilde X, \tilde Y)$. Moreover, the proof of existence of $\II_{H(v)}$ in Proposition \ref{limiting curvature prop} shows that $\II_{H(\tilde v)}$ is the limit of the fundamental forms $\II_{S_{\tilde \gamma(r)}(r)}$ of spheres centered at $\tilde \gamma(r)$ with radius $r$ as $r$ tends to infinity, as previously remarked.

The second fundamental forms of spheres and horospheres both satisfy a revealing ordinary differential equation. Let $\gamma$ be a geodesic in $\tilde M$ and let $X$ be a unit normal parallel vector field along $\gamma$. Moreover suppose $J$ is a Jacobi field along $\gamma$ for which $J(0) = X$ and
\[
	J(r) = O(1) \qquad \text{ for } r \leq 0.
\]
Note,
\[
	\langle \II_{H(-\gamma')}(X, X), -\gamma' \rangle = \frac{\langle \frac{D}{dr} J, X \rangle}{\langle J, X \rangle}.
\]
Differentiating the right hand side shows that $\langle \II_{H(-\gamma')}(X, X), -\gamma' \rangle$ satisfies the ordinary differential equation
\begin{equation}\label{ODE}
	\frac{d}{dr} u + K(X,\gamma'(r)) + u^2 = 0
\end{equation}
where $K$ is the sectional curvature of $\tilde M$. 
The same equation is satisfied if we replace $\II_{H(-\gamma'(r))}$ with $\II_{S_{\gamma(0)}(r)}$. To see this, let $J$ and $Y$ be respective angular and radial coordinate vector fields of some spherical coordinates about $\gamma(0)$, defined on a neighborhood of $\gamma(r)$ for $r > 0$. In particular, we choose $Y$ so that $\gamma' = Y$, $J$ restricts to a Jacobi field along $\gamma$ with $J(0) = 0$ and $\frac{D}{dr} J(0) = X(0)$, and
\[
	0 = [J,Y] = \nabla_J Y - \nabla_Y J.
\]
Since $J$ is parallel to $X$ and vanishes uniquely at $\gamma(0)$, $X = J/|J|$. Hence,
\[
	\langle \II_{S_{\gamma(0)}(r)}(X,X), - \gamma' \rangle = -\frac{\langle \nabla_J X,\gamma' \rangle}{\langle J, X \rangle} = \frac{\langle X, \nabla_J Y \rangle}{\langle J, X \rangle} = \frac{\langle X, \nabla_{Y} J \rangle}{\langle J, X \rangle} = \frac{\langle X, \frac{D}{dr} J \rangle}{\langle J, X \rangle},
\]
so similarly satisfies \eqref{ODE}. This ordinary differential equation provides us with means to bound $\langle \II_{H(v)}, v\rangle$ and to compare the second fundamental forms of spheres of large radius to those of horocycles.

\begin{proposition} \label{curvature properties}
The following are true.
\begin{enumerate}
\item If the sectional curvature $K$ of $M$ satisfies bounds $-a^2 \geq K \geq -b^2$ for some nonnegative constants $a$ and $b$, then
\[
	a|X|^2 \leq \langle \II_{H(v)}(X,X), v \rangle \leq b|X|^2
\]
for all $v$.
\item For all $r > 0$,
\[
	0 < \langle \II_{S_{\gamma(0)}(r)}(X,X), -\gamma'(r) \rangle - \langle \II_{H(-\gamma'(r))}(X,X), - \gamma'(r) \rangle \leq r^{-1}|X|^2.
\]
\end{enumerate}
\end{proposition}

\begin{proof} Let $X$ be a unit length, parallel vector field normal to $\gamma$ and set
\[
	u(r) = \langle \II_{H(-\gamma'(r))}(X(r),X(r)), -\gamma'(r) \rangle \qquad r \in \R
\]
and
\[
	v(r) = \langle \II_{S_{\gamma(0)}(r)}(X(r),X(r)), -\gamma'(r) \rangle \qquad r > 0.
\]
Both $u$ and $v$ satisfy \eqref{ODE} as argued above.

By Proposition \ref{limiting curvature prop}, $u \geq 0$ and is uniformly bounded by continuity of $\II_{H(v)}$ and compactness of $SM$. If $u(r_0) > b$ for some $r_0 \in \R$, then
\[
	u'(r) \leq b^2 - u^2(r_0) < 0 \qquad \text{ for } r \leq r_0,
\]
which contradicts boundedness. If $u(r_0) < a$, then
\[
	u'(r) \geq a^2 - u^2(r_0) > 0 \qquad \text{ for } r \leq r_0,
\]
which contradicts nonpositivity. (1) follows.

(2) can be obtained by using the methods in the proof of Proposition \ref{limiting curvature prop}, but it also follows from \eqref{ODE}. Note,
\[
	v'(r) - u'(r) = - (v^2(r) - u^2(r)).
\]
Since $u$ is bounded and the curvature of small spheres is large, $v(r) - u(r) > 0$ for small $r$. Since $v' - u' = 0$ where $v = u$, $v(r) - u(r) > 0$ for all $r > 0$, hence the lower bound in (2). Then,
\[
	v'(r) - u'(r) = -\frac{v(r) + u(r)}{v(r) - u(r)}(v(r) - u(r))^2 \leq -(v(r) - u(r))^2,
\]
which implies the upper bound by an elementary computation.
\end{proof}

\begin{remark} \label{curvature remark}
Part (2) of the proposition above implies the difference between the second fundamental form of a sphere and that a tangential horocycle is always nondegenerate. This provides us with part (2) of Corollary \ref{criteria} from Theorem \ref{main 2}. Part (1) shows that if $K$ is strictly negative, $\langle \II_{H(v)}(X,X), v \rangle$ is strictly positive definite. Hence, part (3) of the corollary.
\end{remark}


\section{The Conclusion of the Proofs of Theorems \ref{main 1} and \ref{main 2}}

In this section we adapt the tools we developed in Section \ref{GEOMETRY} to local coordinates to prove Propositions \ref{main 1 prop} and \ref{main 2 prop}. The respective main results, Theorems \ref{main 1} and \ref{main 2}, follow. Recall we are trying to bound an oscillatory integral of the form
\begin{equation} \label{final osc int}
	\int_{\R^d} \int_{\R^d} a(x,y) e^{\pm i \lambda \phi(x,y)} \, dx \, dy 
\end{equation}
where
\[
	\phi(x,y) = d_\tg(\alpha \tilde x, \tilde y)
\]
and
\[
	a(x,y) = a_{\alpha, \pm}(T,\lambda; x,y).
\]
We have determined much of the behavior of $\phi$ in the last section, and in Section 3, we determined that $\supp a \subset \supp b \times \supp b$, and
\[
	|\partial_x^{\beta_1} \partial_y^{\beta_2} a(x,y)| \leq C_{\beta} e^{C_\beta T}
\]
for multiindices $\beta$, among other things.

After taking the supports of $b$ to be small and perhaps taking a smooth extension of $\Sigma$ in $M$, we assume $\supp b$ is contained inside a ball $B \subset \R^d$ centered at $0$ in our Fermi local coordinates \eqref{fermi coordinates}. Furthermore, we assume the phase function $\phi$ is defined on $2B \times 2B$ with the same center but twice the radius. Fix $(x,y) \in 2B \times 2B$ and let $v_1(x,y)$ and $v_2(x,y)$ are the unit vectors denoting the arriving and departing directions, respectively, of the geodesic in $\tilde M$ starting at $\tilde y \in \tilde \Sigma$ and ending at $\alpha \tilde x \in \alpha \tilde \Sigma$. By abuse of notation, we will also use $v_1$ and $v_2$ to denote their push-forwards to $M$ through the covering map where appropriate.

We fix a constant $\epsilon > 0$ and consider $\alpha \in \Gamma_U \setminus \Gamma_R$ for which
\begin{equation}\label{nonst alpha}
	|\nabla \phi(x,y)| > \epsilon \qquad \text{ for some } (x,y) \in 2B \times 2B,
\end{equation}
where here $\nabla$ is the gradient with respect to the product metric on $\Sigma \times \Sigma$.
By Remark \ref{bounded hessian remark} and \eqref{hessian computation final}, the Hessian of $\Hess_{\Sigma \times \Sigma} \phi$ is a uniformly bounded quadratic form for non-identity $\alpha$. Hence by the mean value theorem, we may restrict $B$ so that
\[
	|\nabla \phi(x,y)| \geq \epsilon/2 \qquad \text{ for all } (x,y) \in 2B \times 2B
\]
for all $\alpha$ satisfying \eqref{nonst alpha}. Since the metric tensor of $\Sigma \times \Sigma$ is nearly the identity at $(0,0)$, by taking $B$ small we ensure that the Euclidean gradient of $\phi$ in local coordinates is bounded below by $\epsilon/4$. The oscillatory integral \eqref{final osc int} is then bounded by a constant multiple of $e^{C_N T} \lambda^{-N}$ for any suitably large $N$ by Part (1) of Lemma \ref{stationary phase lemma}. 

All that remains is the situation where
\begin{equation}\label{small gradient}
	|\nabla \phi| \leq \epsilon \qquad \text{ on } 2B \times 2B.
\end{equation}

Now is when we really capitalize on our ability to take $R$ large and restrict $B$ and $U$. Recall that $\nabla_{x,y}^2 \phi$ is the Euclidean Hessian matrix of $\phi$ in the variables $x$ and $y$. We eventually want to show
\begin{equation} \label{hessian at 0}
	\nabla_{x,y}^2 \phi(x,y) = \begin{bmatrix} \nabla_{x}^2 \phi(0,0) & 0 \\ 0 & \nabla_y^2 \phi(0,0) \end{bmatrix} + E(x,y)
\end{equation}
for all $\alpha \in \Gamma_U \setminus \Gamma_R$, where $E$ is an error matrix whose entries are controlled by an adjustably small constant uniform in $\alpha$. By \eqref{hessian in coordinates} and since the Christoffel symbols of the product metric on $\Sigma \times \Sigma$ vanish at $(0,0)$, we may restrict the support of $b$ so that
\begin{align*}
	\partial_{x_i} \partial_{x_j} \phi(x,y) &= \Hess_{\Sigma \times \Sigma} \phi(\partial_{x_i}, \partial_{x_j}) \\
	\partial_{y_i} \partial_{y_j} \phi(x,y) &= \Hess_{\Sigma \times \Sigma} \phi(\partial_{y_i}, \partial_{y_j}), \qquad \text{ and } \\
	\partial_{x_i} \partial_{y_j} \phi(x,y) &= \Hess_{\Sigma \times \Sigma} \phi(\partial_{x_i}, \partial_{y_j})
\end{align*}
modulo some small, controllable error terms for $i,j = 1,\ldots, d$. Hence, it suffices to show
\begin{align} \label{hessian at 0'}
	\Hess_{\Sigma \times \Sigma} \phi(x,y)(\partial_{x_i},\partial_{x_j}) &= \Hess_{\Sigma \times \Sigma} \phi(0,0)(\partial_{x_i},\partial_{x_j}), \\
	\nonumber \Hess_{\Sigma \times \Sigma} \phi(x,y)(\partial_{y_i},\partial_{y_j}) &= \Hess_{\Sigma \times \Sigma} \phi(0,0)(\partial_{y_i},\partial_{y_j}), \qquad \text{ and } \\
	\nonumber \Hess_{\Sigma \times \Sigma} \phi(x,y)(\partial_{x_i},\partial_{y_j}) &= 0
\end{align}
modulo small, controllable error terms which are bounded independently of $\alpha$. Note the third line follows by taking $R$ large and invoking part (2) of Lemma \ref{hessian computation}.

Fix indices $i$ and $j$. We claim that the diameter of the set
\[
	\{ \Hess_{\Sigma \times \Sigma}\phi(x,y)(\partial_{x_i}, \partial_{x_j}) : x,y \in B \}
\]
can be controlled by taking $B$ and $\epsilon$ small and $R$ large. By part (2) of Proposition \ref{curvature properties}
\[
	\langle \II_{S_{\tilde y}(d_\tg)}(\partial_{x_i}', \partial_{x_j}'), -v_1 \rangle = \langle \II_{H(-v_1)}(\partial_{x_i}', \partial_{x_j}'), -v_1 \rangle 
\]
modulo an error term controllable by taking $R$ large. Hence by \eqref{hessian computation final}, we have
\[
	\Hess_{\Sigma \times \Sigma}\phi(\partial_{x_i}, \partial_{x_j}) = \langle \II_{\Sigma}(\partial_{x_i}, \partial_{x_j}) , v_1 \rangle - \langle \II_{H(-v_1)}(\partial_{x_i}', \partial_{x_j}') , v_1 \rangle
\]
modulo controllable error terms. The diameter of the set of values achieved by the first term on the right is controlled by taking $v_1$ close to normal, i.e. by taking $\epsilon$ small, and similarly for the second term. The first line of \eqref{hessian at 0'} follows. The second line follows similarly. We now have \eqref{hessian at 0} and are ready to prove our propositions.

\begin{proof}[Proof of Proposition \ref{main 1 prop}] 
	We will select $d$ coordinates in which to use the method of stationary phase in order to obtain the desired bound
	\begin{equation} \label{final osc int 1}
		\left|\int_{\R^d} \int_{\R^d} a(x,y) e^{\pm i \lambda \phi(x,y)} \, dx \, dy \right| \lesssim e^{CT} \lambda^{-d/2}
	\end{equation}
	for \eqref{final osc int}. By an orthogonal change of variables on $\R^d$, we may take $\partial_{x_i}$ for $i = 1,\ldots, d$ to align with the principal directions of $\Sigma$ at $0$. Now,
	\begin{equation}\label{hess xx}
		\partial_{x_i}\partial_{x_j} \phi(0,0) = \begin{cases} \langle \kappa_i, v_1 \rangle - \langle \II_{H(-v_1)}(\partial_{x_i}', \partial_{x_i}'), v_1 \rangle & i = j \\
		- \langle \II_{H(-v_1)}(\partial_{x_i}', \partial_{x_j}'), v_1 \rangle & i \neq j
		\end{cases}
	\end{equation}
	and
	\[
		\partial_{y_i}\partial_{y_j} \phi(0,0) = \begin{cases} \langle \kappa_i, -v_2 \rangle - \langle \II_{H(v_2)}(\partial_{y_i}', \partial_{y_i}'), -v_2 \rangle & i = j \\
		- \langle \II_{H(v_2)}(\partial_{y_i}', \partial_{y_j}'), -v_2 \rangle & i \neq j
		\end{cases}
	\]
	where $\kappa_i = \II_\Sigma(\partial_{x_i}, \partial_{x_i})$ is the $i$th principal curvature vector. By part (1) of Proposition \ref{curvature properties} and since the sectional curvature of $M$ is strictly negative, we can take $\epsilon$ small enough so that
	\begin{align}
		\label{H pos def} \sum_{i,j = 1}^d \langle \II_{H(-v_1)}(\partial_{x_i}', \partial_{x_j}'), -v_1 \rangle \xi_i \xi_j &\geq c|\xi|^2 \qquad \text{ and } \\
		\nonumber \sum_{i,j = 1}^d \langle \II_{H(v_2)}(\partial_{y_i}', \partial_{y_j}'), v_2 \rangle \xi_i \xi_j &\geq c|\xi|^2
	\end{align}
	for all $\xi \in \R^d$ and for some positive constant $c$ at $(x,y) = (0,0)$. By taking $\epsilon$ small, we can ensure $v_1$ and $v_2$ are close and that 
	\begin{equation}\label{selection}
		\langle \kappa_i, v_1 \rangle \geq -c/2 \qquad \text{ or } \qquad \langle \kappa_i, -v_2 \rangle \geq -c/2
	\end{equation}
	for each $i = 1,\ldots, d$. We pick coordinates $z = (z_1,\ldots, z_d)$ where $z_i = x_i$ if $\langle \kappa_i, v_1 \rangle \geq -c/2$ and $z_i = y_i$ if $\langle \kappa_i, -v_2 \rangle \geq -c/2$. By reordering, assume that
	\[
		z = (x_1, \ldots, x_\ell, y_{\ell+1}, \ldots, y_d)
	\]
	for some $\ell \in \{0,1,\ldots,d\}$, and let $w = (y_1,\ldots, y_\ell, x_{\ell+1},\ldots, x_d)$ be the complimentary coordinates. We bound the left side of \eqref{final osc int 1} by
	\[
		\int_{\R^d} \left| \int_{\R^d} a(x,y) e^{\pm i \lambda \phi(x,y)} \, dz \right| \, dw
	\]
	and use the method of stationary phase on the inner integral to obtain the desired bound. By \eqref{hessian at 0},
	\[
		\nabla_z^2 \phi(x,y) = \begin{bmatrix} \nabla_{x_1,\ldots,x_\ell}^2 \phi(0,0) & 0 \\
		0 & \nabla_{y_{\ell+1}, \ldots, y_d}^2 \phi(0,0) \end{bmatrix} + E(x,y).
	\]
	Now by \eqref{hess xx}, \eqref{H pos def}, and our selection of coordinates by \eqref{selection},
	\[
		|\nabla_{x_1,\ldots,x_\ell}^2 \phi(0,0) \xi| \geq \frac{c}{2} |\xi| \qquad \text{ for all } \xi \in \R^\ell
	\]
	and similarly
	\[
		|\nabla_{y_{\ell+1}, \ldots, y_d}^2 \phi(0,0) \xi| \geq \frac{c}{2} |\xi| \qquad \text{ for all } \xi \in \R^{d-\ell}.
	\]
	Hence if $E(x,y)$ is made small enough,
	\[
		|\nabla_{z}^2 \phi(x,y) \xi| \geq \frac{c}{4} |\xi| \qquad \text{ for all } \xi \in \R^d, \ x,y \in 2B.
	\]
	The proposition follows after an application of Lemma \ref{stationary phase lemma}.
\end{proof}

\begin{proof}[Proof of Proposition \ref{main 2 prop}]
	Let $v$ be the normal vector to $\Sigma$ which points in a similar direction to $v_1$ and $v_2$. By the hypotheses \eqref{main 2 hypotheses}, we select two subspaces $V_1$ and $V_2$ of $T\Sigma$, with respective dimensions $\ell_1$ and $\ell_2$ with $\ell_1 + \ell_2 = n$, and on which the restriction of the quadratic form $\langle \II_\Sigma - \II_{H(-v)}, v \rangle$ to $V_1$ and the restriction of $\langle \II_\Sigma - \II_{H(v)}, -v \rangle$ to $V_2$ are nondegenerate. In particular, select local coordinates $(x_1,\ldots, x_{\ell_1})$ of $V_1$ such that $\partial_{x_1}, \ldots, \partial_{x_{\ell_1}}$ forms an orthonormal basis at $0$ at which
	\[
		|\langle \II_\Sigma(\partial_{x_i}, \partial_{x_i}) - \II_{H(-v)}(\partial_{x_i}, \partial_{x_i}), v \rangle| \geq 4c \qquad \text{ for } i = 1,\ldots, \ell_1
	\]
	for some positive constant $c$
	and
	\[
		\langle \II_\Sigma(\partial_{x_i}, \partial_{x_j}) - \II_{H(-v)}(\partial_{x_i}, \partial_{x_j}), v \rangle = 0 \qquad \text{ for } i \neq j.
	\]
	By ensuring $\epsilon$ in \eqref{small gradient} is sufficiently small, we take
	\[
		|\partial_{x_i} \partial_{x_i} \phi(0,0)| = |\langle \II_\Sigma(\partial_{x_i}, \partial_{x_i}) - \II_{H(-v_1)}(\partial_{x_i}', \partial_{x_i}'), v_1 \rangle| \geq 2c \qquad \text{ for } i = 1,\ldots, \ell_1
	\]
	and
	\[
		|\partial_{x_i} \partial_{x_j} \phi(0,0)| = |\langle \II_\Sigma(\partial_{x_i}, \partial_{x_j}) - \II_{H(-v_1)}(\partial_{x_i}', \partial_{x_j}'), v_1 \rangle| \leq c/8n \qquad \text{ for } i \neq j.
	\]
	We similarly select a parametrization $(y_1,\ldots, y_{\ell_2})$ of $V_2$ for which
	\[
		|\partial_{y_i} \partial_{y_i} \phi(0,0)| = |\langle \II_\Sigma(\partial_{y_i}, \partial_{y_i}) - \II_{H(v_2)}(\partial_{y_i}', \partial_{y_i}'), -v_2 \rangle| \geq 2c \qquad \text{ for } i = 1,\ldots, \ell_2
	\]
	and
	\[
		|\partial_{y_i} \partial_{y_j} \phi(0,0)| = |\langle \II_\Sigma(\partial_{y_i}, \partial_{y_j}) - \II_{H(v_2)}(\partial_{y_i}', \partial_{y_j}'), -v_2 \rangle| \leq c/8n \qquad \text{ for } i \neq j.
	\]
	By bounding each of the entries of $E(x,y)$ in \eqref{hessian at 0} by $c/8n$, the $n \times n$ Hessian matrix $\nabla_{x_1,\ldots,x_{\ell_1}, y_1,\ldots, y_{\ell_2}}^2 \phi(x,y)$ has diagonal terms whose absolute values are bounded below by $c$, and off-diagonal terms bounded by $c/4n$. It follows that
	\[
		|\nabla_{x_1,\ldots,x_{\ell_1}, y_1,\ldots, y_{\ell_2}}^2 \phi(x,y) \xi| \geq \frac{c}{2}|\xi| \qquad \text{ for all } \xi \in \R^{n}, \ x,y \in 2B.
	\]
	This and Lemma \ref{stationary phase lemma} show us
	\[
		\left| \idotsint a(x,y) e^{i\pm \lambda \phi(x,y)} \, dx_1 \cdots dx_{\ell_1} dy_1 \cdots dy_{\ell_2} \right| \lesssim e^{CT} \lambda^{-n/2}
	\]
	uniformly over the remaining variables $x_{\ell+1},\ldots,x_{n-1}$ and $y_{\ell+1}, \ldots, y_{n-1}$. The integral in \eqref{final osc int} hence satisfies the same bounds.
\end{proof}


\section{Appendix}

\subsection{Exponential bounds on mixed derivatives}

The following proposition allows us to obtain exponential bounds on mixed derivatives of functions $f(x,y)$ in $C^\infty(\tilde M \times \tilde M)$ if we are only provided with exponential bounds on pure derivatives in both variables. We use this to obtain bounds on the mixed derivatives of the amplitudes in Lemma \ref{kernel lemma}.

\begin{proposition} \label{appendix prop} Let $(M,g)$ be a compact, $n$-dimensional, boundaryless Riemannian manifold with nonpositive sectional curvature and let $(\tilde M, \tilde g)$ denote the universal cover of $M$ equipped with the pullback metric. Let $f : \tilde M \times \tilde M \to \R$ be a function satisfying bounds
\[
	|\Delta_x^j f(x,y)| \leq C_j e^{C_j d_\tg(x,y)} \qquad \text{ and } \qquad	
	|\Delta_y^k f(x,y)| \leq C_k e^{C_j d_\tg(x,y)}
\]
where $d_\tg(x,y) \geq 1$. Then,
\[
	|\Delta_x^j \Delta_y^k f(x,y)| \leq C_{j,k} e^{C_{j,k} d_\tg(x,y)} \qquad \text{ for } d_\tg(x,y) \geq 1,
\]
where the constants $C_{j,k}$ depend only on the constants $C_j$ and $C_k$ and the manifold.
\end{proposition}

\begin{proof}
Fix $x_0$ and $y_0$ in $\tilde M$ and fix a smooth function $\beta \in C_0^\infty(\R,[0,1])$ equal to $1$ near $0$ and supported in $(-\inj M, \inj M)$. Then let
\[
	F(x,y) = \beta(d_\tg(x,x_0)) \beta(d_\tg(y,y_0)) f(x,y).
\]
Note
\begin{equation}\label{appendix 2}
	|\Delta_x^j F(x,y)| \leq C_j' e^{C_j'd_\tg(x_0,y_0)} \qquad \text{ and } \qquad |\Delta_y^k F(x,y)| \leq C_k' e^{C_k'd_\tg(x_0,y_0)}
\end{equation}
by \eqref{distance bounds} for constants $C_j'$ and $C_k'$ which are independent of $x$, $y$, $x_0$, and $y_0$.
The cutoffs allow us to interpret $F$ as a function on $M \times M$.
By Sobolev embedding,
\begin{align}
\nonumber	|\Delta_x^j \Delta_y^k f(x_0,y_0)| &\leq \| \Delta_x^j \Delta_y^k F(x,y) \|_{L^\infty(M \times M)}\\
\label{appendix 1}	&\leq C \| (I - \Delta_x - \Delta_y)^{n+1} \Delta_x^j \Delta_y^k F(x,y) \|_{L^2(M\times M)}
\end{align}
where we understand $\Delta_x + \Delta_y$ as the Laplace-Beltrami operator on the product manifold $M\times M$. It follows $e_p(x) e_q(y)$ for $p,q = 0,1,2,\ldots$ form an orthonormal basis of eigenfunctions on $M \times M$ with
\[
	(\Delta_x + \Delta_y)e_p(x) e_q(y) = -(\lambda_p^2 + \lambda_q^2) e_p(x) e_q(y).
\]
We use the shorthand
\[
	\hat F(p,q) = \int_M \int_M F(x,y) \overline{e_p(x) e_q(y)} \, dx \, dy
\]
and write
\begin{align*}
	&\| (I - \Delta_x - \Delta_y)^{n+1} \Delta_x^j \Delta_y^k F(x,y) \|_{L^2(M\times M)}^2\\
	&= \sum_{p,q} (1 + \lambda_p^2 + \lambda_q^2)^{2n+2} \lambda_p^{4j} \lambda_q^{4k} |\hat F(p,q)|^2\\
	&\leq \sum_{p,q} (1 + \lambda_p^{4(n+j+k+1)} + \lambda_q^{4(n+j+k+1)}) |\hat F(x,y)|^2\\
	&= \|F\|_{L^2(M\times M)}^2 + \| \Delta_x^{n+j+k+1} F\|_{L^2(M\times M)}^2 + \| \Delta_y^{n+j+k+1} F\|_{L^2(M\times M)}^2.
\end{align*}
Finally,
\begin{align*}
	&\|F\|_{L^2(M\times M)}^2 + \| \Delta_x^{n+j+k+1} F\|_{L^2(M\times M)}^2 + \| \Delta_y^{n+j+k+1} F\|_{L^2(M\times M)}^2 \\
	&\leq \vol(M)^2 \left(\|F\|_{L^\infty(M\times M)}^2 + \| \Delta_x^{n+j+k+1} F\|_{L^\infty(M\times  M)}^2 + \| \Delta_y^{n+j+k+1} F\|_{L^\infty(M\times M)}^2 \right),
\end{align*}
and the proposition follows from \eqref{appendix 2}.
\end{proof}

\subsection{A Stationary Phase Lemma}

The following stationary phase lemma helps us obtain uniform bounds on \eqref{final osc int} in both the proofs of Propositions \ref{main 1 prop} and \ref{main 2 prop}.


\begin{lemma}\label{stationary phase lemma}
	Let
	\[
		I(\lambda) = \int_{\R^n} a(x) e^{i\lambda \phi(x)} \, dx
	\]
	where $a$ is a smooth function on $\R^n$ with support contained in the unit ball $B = \{x : |x| \leq 1\}$, and where $\phi$ is a smooth function on $\sqrt 2B = \{x : |x| \leq \sqrt 2\}$.
	\begin{enumerate}
	
		\item If $|\nabla \phi(x,y)| \geq c$ on $B$ for some $c > 0$, then
		\[
			|I(\lambda)| \leq C_N \lambda^{-N} \qquad \text{ for } \lambda \geq 1
		\]
		for $N = 1,2,\ldots$.
		
		\item If
		\[
			\left| (\nabla^2 \phi) \xi \right| \geq c |\xi| \qquad \text{ for all } \xi \in \R^n
		\]
		on $\sqrt 2B$ for some $c > 0$, then
		\[
			|I(\lambda)| \leq C \lambda^{-n/2} \qquad \text{ for } \lambda \geq 1.
		\]
	\end{enumerate}
	
	In both situations (1) and (2), the constants $C$ and $C_N$ are polynomials in $c^{-1}$ and $\sup_{B} |\partial_x^\beta a|$ and $\sup_{B} |\partial_x^\beta \phi|$ for finitely many multiindices $\beta$.
\end{lemma}


\begin{proof}
Let $\gamma : [0,\ell] \to \R^n$ be a unit speed curve in $\sqrt 2 B$ where
\[
	\nabla \phi(\gamma(t)) \neq 0 \qquad \text{ for } t \in (0,\ell)
\]
and
\[
	\gamma'(t) = \frac{\nabla |\nabla \phi|}{|\nabla |\nabla \phi||}.
\]
Setting $\gamma(0) = x_0$ and $\gamma(\ell) = x_1$, the mean value theorem gives us a time $t \in (0,\ell)$ at which
\begin{align}
	\nonumber |\nabla \phi(x_1)| - |\nabla \phi(x_0)| &= \ell \frac{d}{dt} |\nabla |\nabla \phi(\gamma(t))|| \\
	\nonumber &= \ell \frac{d}{dt} \left|\nabla^2 \phi(\gamma(t)) \frac{\nabla \phi(\gamma(t))}{|\nabla \phi(\gamma(t))|}\right| \\
	\nonumber &\geq \ell c \\
	\label{mvt} &\geq c|x_1 - x_0|.
\end{align}
If $\phi$ has a critical point at some $x_0$ in $\sqrt 2B$, since $\nabla^2 \phi(x_0)$ is a linear isomorphism from $\R^n \to \R^n$, there exist such flow lines of $\nabla |\nabla \phi|$ in every direction starting at $x_0$. Moreover by \eqref{mvt}, $|\nabla \phi| \neq 0$ on this neighborhood minus the point $x_0$. By an open-closed argument, there exists such a flow line connecting $x_0$ to any other point $x \in \sqrt 2B$, and we conclude
\[
	|\nabla \phi(x)| \geq c|x - x_0| \qquad \text{ for all } x \in 2B
\]
from \eqref{mvt}. The desired bound on $I(\lambda)$ follows from this estimate of $|\nabla \phi(x)|$ and careful inspection of the proof of \cite[Proposition 4.1.2]{Hang}.

On the other hand, if there are no critical points of $\phi$ in $\sqrt2B$, we have
\[
	|\nabla |\nabla \phi|| = \left| \nabla^2 \phi \frac{\nabla \phi}{|\nabla \phi|} \right| \geq c > 0,
\]
and hence $|\nabla \phi|$ has no critical points on $\sqrt2B$. In particular, $|\nabla \phi|$ attains a minimum on $B$ only on the boundary. Select such a point $x_0$ on $\partial B$ and take a unit-speed curve $\gamma$ with $\gamma(0) = x_0$ and
\[
	\gamma'(t) = -\frac{\nabla |\nabla \phi|}{|\nabla |\nabla \phi||}.
\]
By the same argument as before,
\[
	|\nabla \phi(x_0)| - |\nabla \phi(\gamma(t))| \geq c t \qquad \text{ for all } t > 0.
\]
Hence, $\gamma(t)$ never intersects $B$ for $t > 0$. Moreover since $|\nabla \phi|$ is bounded below on $\sqrt 2B$, $\gamma$ must intersect the boundary $\partial(\sqrt 2B)$ at some point $x_1$ at some time $\ell$. Hence,
\[
	\inf_{B} |\nabla \phi| = |\nabla \phi(x_0)| \geq c \ell \geq c(\sqrt 2 - 1).
\]
We obtain a bound of
\[
	|I(\lambda)| \leq C_N \lambda^{-N} \qquad N = 1,2,\ldots
\]
where the constants $C_N$ are polynomials of the desired quantities by careful inspection of the proof of \cite[Proposition 4.1.1]{Hang}.
\end{proof}



\bibliography{references}{}

\def\cprime{$'$} \def\cprime{$'$}
\begin{thebibliography}{Wym17b}

\bibitem[B{\'e}r77]{Berard}
P.~H. B{\'e}rard.
\newblock On the wave equation on a compact {R}iemannian manifold without
  conjugate points.
\newblock {\em Math. Z.}, 155(3):249--276, 1977.

\bibitem[CG]{CanGal}
Y.~Canzani and J.~Galkowski.
\newblock On the growth of eigenfunction averages: microlocalization and
  geometry.
\newblock preprint.

\bibitem[CGT17]{toth}
Y.~Canzani, J.~Galkowski, and J.~A. Toth.
\newblock Averages of eigenfunctions over hypersurfaces.
\newblock Preprint, 2017.

\bibitem[CS15]{CSPer}
X.~Chen and C.~D. Sogge.
\newblock On integrals of eigenfunctions over geodesics.
\newblock {\em Proc. Amer. Math. Soc.}, 143(1):151--161, 2015.

\bibitem[dC92]{doCarmo}
Manfredo~Perdig{\~a}o do~Carmo.
\newblock {\em Riemannian geometry}.
\newblock Mathematics: Theory \& Applications. Birkh\"auser Boston, Inc.,
  Boston, MA, 1992.
\newblock Translated from the second Portuguese edition by Francis Flaherty.

\bibitem[Goo83]{Good}
A.~Good.
\newblock {\em Local analysis of {S}elberg's trace formula}, volume 1040 of
  {\em Lecture Notes in Mathematics}.
\newblock Springer-Verlag, Berlin, 1983.

\bibitem[Hej82]{Hej}
D.~A. Hejhal.
\newblock Sur certaines s\'eries de {D}irichlet associ\'ees aux g\'eod\'esiques
  ferm\'ees d'une surface de {R}iemann compacte.
\newblock {\em C. R. Acad. Sci. Paris S\'er. I Math.}, 294(8):273--276, 1982.

\bibitem[HR18]{HGtorus}
Hamid Hezari and Gabriel Riviere.
\newblock Equidistribution of toral eigenfunctions along hypersurfaces.
\newblock 01 2018.

\bibitem[Rez15]{Rez}
A.~Reznikov.
\newblock A uniform bound for geodesic periods of eigenfunctions on hyperbolic
  surfaces.
\newblock {\em Forum Math.}, 27(3):1569--1590, 2015.

\bibitem[Sog14]{Hang}
C.~D. Sogge.
\newblock {\em Hangzhou lectures on eigenfunctions of the {L}aplacian}, volume
  188 of {\em Annals of Mathematics Studies}.
\newblock Princeton University Press, Princeton, NJ, 2014.

\bibitem[Sog17]{SFIO}
C.~D. Sogge.
\newblock {\em Fourier integrals in classical analysis}, volume 210 of {\em
  Cambridge Tracts in Mathematics}.
\newblock Cambridge University Press, Cambridge, 2nd edition, 2017.

\bibitem[STZ11]{STZ}
C.~D. Sogge, J.~A. Toth, and S.~Zelditch.
\newblock About the blowup of quasimodes on {R}iemannian manifolds.
\newblock {\em J. Geom. Anal.}, 21(1):150--173, 2011.

\bibitem[SXZ17]{Gauss}
C.~D. Sogge, Y.~Xi, and C.~Zhang.
\newblock Geodesic period integrals of eigenfunctions on {R}iemannian surfaces
  and the {G}auss-{B}onnet theorem.
\newblock {\em Camb. J. Math.}, 5(1):123--151, 2017.

\bibitem[SY94]{redbook}
R.~Schoen and S.~T. Yau.
\newblock {\em Lectures on Differential Geometry}.
\newblock Int. Press, 1994.

\bibitem[SZ02]{SZDuke}
C.~D. Sogge and S.~Zelditch.
\newblock Riemannian manifolds with maximal eigenfunction growth.
\newblock {\em Duke Math. J.}, 114(3):387--437, 2002.

\bibitem[Wym17a]{emmett2}
E.~Wyman.
\newblock Explicit bounds on integrals of eigenfunctions over curves in
  surfaces of nonpositive curvature.
\newblock preprint, 2017.

\bibitem[Wym17b]{emmett1}
E.~Wyman.
\newblock Integrals of eigenfunctions over curves in surfaces of nonpositive
  curvature.
\newblock preprint, 2017.

\bibitem[Wym18]{emmett3}
Emmett~L. Wyman.
\newblock Looping directions and integrals of eigenfunctions over submanifolds.
\newblock {\em The Journal of Geometric Analysis}, Jun 2018.

\bibitem[Zel92]{ZelK}
S.~Zelditch.
\newblock Kuznecov sum formulae and {S}zeg{\H o}\ limit formulae on manifolds.
\newblock {\em Comm. Partial Differential Equations}, 17(1-2):221--260, 1992.

\end{thebibliography}
\bibliographystyle{alpha}

\end{document}